\documentclass[reqno, a4paper]{amsart}
\usepackage{amssymb, amscd, amsfonts, amsthm}
\usepackage{thmtools}
\usepackage{mathtools}
\usepackage{tikz-cd}
\usepackage{enumitem}
\usepackage{hyperref}
\usepackage{cleveref}

\declaretheorem[name=Theorem,refname={Theorem,Theorems},Refname={Theorem,Theorems},numberwithin=section,]{theorem}
\declaretheorem[name=Proposition,refname={Proposition,Propositions},Refname={Proposition,Propositions},sibling=theorem,]{proposition}
\declaretheorem[name=Lemma,refname={Lemma,Lemmas},Refname={Lemma,Lemmas},sibling=theorem,]{lemma}
\declaretheorem[name=Corollary,refname={Corollary,Corollaries},Refname={Corollary,Corollaries},sibling=theorem,]{corollary}

\declaretheorem[name=Definition,refname={Definition,Definitions},Refname={Definition,Definitions},sibling=theorem,style=definition,]{definition}
\declaretheorem[name=Notation,refname={Notation,Notations},Refname={Notation,Notations},sibling=theorem,style=definition,]{notation}

\declaretheorem[name=Remark,refname={Remark,Remarks},Refname={Remark,Remarks},sibling=theorem,style=remark,]{remark}

\newcommand{\luk}{\L u\-ka\-s\-ie\-wicz}

\DeclareMathOperator{\interval}{[0,1]}
 
\DeclareMathOperator{\maxspec}{{\boldsymbol{\mu}}}

\newcommand{\m}{\mathfrak{m}}
 
\newcommand{\Bms}{\mathsf{Bms}}
\newcommand{\uSlg}{\mathsf{uS{\ell}g}}
\renewcommand{\lg}{\mathsf{{\ell}g}}
\newcommand{\ulg}{\mathsf{u{\ell}g}}
\newcommand{\SMV}{\mathsf{SMV}}
\newcommand{\Set}{\mathsf{Set}}

\newcommand{\B}{\mathcal{B}}
\renewcommand{\SS}{\mathcal{S}}

\newcommand{\Z}{\mathbb{Z}}

\newcommand{\LCM}{\mathsf{LCM}}

\newcommand*\cocolon{%
	\nobreak
	\mskip6mu plus1mu
	\mathpunct{}%
	\nonscript
	\mkern-\thinmuskip
	{:}%
	\mskip2mu
	\relax
} 

 \title[Unital Specker $\ell$-groups and boolean multispaces]
{  
Unital Specker $\ell$-groups and boolean multispaces
 }

\author{Marco Abbadini}

\address[Marco Abbadini]{Research Institute in Mathematics and Physics\\
Université ca\-tho\-lique de Louvain\\
Chem.\ du Cyclotron 2\\
1348 Ottignies-Louvain-la-Neuve\\
Belgium}
\email{marco.abbadini@uclouvain.be}

\author{Daniele Mundici}

\address[Daniele Mundici]{Department of
Mathematics and Computer Science  ``Ulisse Dini''\\
University of Florence\\
Viale Morgagni 67/A\\
I-50134 Florence \\
Italy}
\email{daniele.mundici@unifi.it }

\thanks{The first author was funded by UK Research and Innovation (UKRI) under the UK government’s Horizon Europe funding guarantee (grant number EP/Y015029/1, Project ``DCPOS'') during his affiliation at the University of Birmingham and by an FSR Incoming Postdoctoral Fellowship during his affiliation at the Université catholique de Louvain.}

\date{}

\begin{document}



\begin{abstract}
	As a topological generalization of the notion of a multiset,
	a {\it boolean multispace} is a boolean space 
	$X$ with a continuous function
	$u\colon X\to \Z_{>0}$, where $\Z_{>0}=\{1,2,\dots\}$
	has the discrete topology.
	In this paper the category of boolean multispaces 
	and continuous multiplicity-decreasing
	morphisms with respect to the divisibility order
	is shown to be dually equivalent to the category of
	unital Specker $\ell$-groups 
	and unital $\ell$-homomorphisms.
	This result extends Stone duality, because
	unital Specker $\ell$-groups whose distinguished unit is
	singular are equivalent to boolean algebras.
	Boolean multispaces, in turn, are categorically
	equivalent to the Priestley duals
	of the MV-algebras corresponding to
	unital Specker $\ell$-groups 
	via the $\Gamma$ functor. Via duality, we show that the category 
	of unital Specker $\ell$-groups has finite colimits and finite products, but 
	lacks some countable copowers and equalizers.
\end{abstract}

\maketitle

\hfill{\large \it Honoring Hilary Priestley}

\section{Introduction}

\begin{quote}
{\it  \small
\dots and one may say that
the invention of functors 
is one of the main goals 
of modern mathematicians, 
and one which usually yields the most
startling
results.}

\footnotesize{\hfill J. Dieudonn\'e,  \cite[p.~236]{die}}
\end{quote}

\medskip
\noindent
MV-algebras were introduced by 
C.C.\ Chang in his  paper \cite{cha58}.
By definition, an MV-algebra $A$ is a structure
$(A,0, \neg,\oplus)$ satisfying the equations
$x\oplus (y\oplus z)= (x\oplus y)\oplus z$, \,\,
$x \oplus 0  =   x$,\,\,
$x \oplus \neg 0 = \neg 0$, \,\,
$\neg \neg x = x$\, and, characteristically,  
\begin{equation}
\label{equation:noncanonical}
x\oplus \neg(x\oplus \neg y)=y\oplus \neg(y\oplus \neg x).
\end{equation}
These equations are 
a terse equivalent reformulation of  Chang's
original axiomatization in \cite{cha58}.
See \cite[\S 1.7]{cigdotmun}.
As shown in  \cite{kol}, the
commutativity of $\oplus$ 
follows from the above five equations.

From Chang's completeness theorem   \cite{cha59}
it follows  that an equation
holds in the unit real interval  $\interval$ 
equipped with the constant 0 and the \luk\ implication
$x \to_{\mbox{\tiny \L}}  y \,\,\,(=  \neg x \oplus y)$
if and only if it holds in every MV-algebra.  See \cite[Theorem 2.5.3]{cigdotmun}.
For short,  the algebra $(\interval, 0, \to_{\mbox{\tiny \L}})$
is a term-equivalent variant of the canonical generator 
$\interval=(\interval, 0, \neg, \oplus)$ of the equational class 
of MV-algebras.

As shown in   \cite{mun-jsl},
equation \eqref{equation:noncanonical} is deeply related to the
continuity of implication in a
$\interval$-valued logic. 
In the terminology of \cite{gehpri},
this identity is not ``canonical''. This
makes it difficult to extend to MV-algebras 
Priestley duality,  \cite{pri1,pri2,pri3}. 
 
In the present paper we consider a special class of 
MV-algebras, called  Specker MV-algebras. They are
the counterpart of  unital Specker $\ell$-groups \cite{bez,
con,condar,nob}
via the  $\Gamma$ functor, \cite{mun-jfa}.
By  \cite[Theorem~3.16]{condar}, Specker  $\ell$-groups
are uniquely determined by their underlying lattice structure together with
the identity element.
By  Lemma \ref{lemma:mv-representation},
Specker MV-algebras are just finite products  $\prod_i B_i\otimes \, $\L$_{n(i)}$
of tensor products of boolean algebras $B_i$ by finite \luk\
chains  \L$_{n(i)}$.

As a topological generalization of the notion of a multiset,
a {\it boolean multispace} is a boolean space,
(i.e., a totally disconnected
compact Hausdorff space)  $X$
equipped  with a continuous function
$u \colon  X\to \Z_{>0}$, where $\Z_{>0}=\{1,2,\dots\}$
has  the discrete topology. 
In Theorem \ref{theorem:duality} it is 
shown  that boolean multispaces and   
continuous multiplicity-decreasing
morphisms with respect to the divisibility order
are  dually equivalent  to 
unital Specker $\ell$-groups 
and  unital $\ell$-homomorphisms.
Since, as proved in \cite{mun-tac},  boolean spaces
are dual to unital Specker $\ell$-groups whose distinguished unit is singular,
this result extends Stone duality.
See Corollary \ref{corollary:stone} for details.
Using  the $\Gamma$ functor
we obtain that  boolean multispaces
are categorically  equivalent to the Priestley duals
of Specker MV-algebras. See
Proposition \ref{proposition:priestley}.
 This result should be compared with the  duality obtained in
the paper   \cite{fusgehgoomar}  for  the class of {\it all} MV-algebras, 
  where the dual of an MV-algebra $A$
  is an enriched Priestley space on the set
  of prime lattice ideals of $A$.
  
Section \ref{section:adjoint} is devoted to the construction
of a functor  $\mathcal B$ that, together with $\mathcal S$,
is part of the duality.

  In Sections 
  	\ref{section:sei}--\ref{section:otto}
it is shown, among other things,
  that 
boolean multispaces are finitely complete and have finite 
coproducts, but---differently from boolean spaces---lack
 some countable powers and some coequalizers.
 
In the final Section \ref{section:related-work}
we briefly discuss related work in the literature on multisets.

\medskip
\noindent
To the best of our knowledge, the present paper is the first one
where  a duality is constructed between a category
of   topological  multispaces  
and a pre-existing,  time-honored
 category of algebraic structures, namely unital Specker $\ell$-groups and their unital 
 $\ell$-homomorphisms.  Unital and  non-unital 
 Specker $\ell$-groups have  well-known representation theorems
in the literature. See, e.g., \cite[\S 13.5]{bkw}, and \cite{bez} and references therein.

 Last but not least, this paper  gives an answer to Problem 5.5  in 
\cite{mun-sl} concerning the construction of 
MV-algebraic  {\it categories}
of Ulam-R\'enyi games on infinite search spaces.  As a matter of fact, the MV-algebras of
compact Ulam-R\'enyi games  (\cite[Definition 5.2]{mun-sl})
equipped with MV-algebraic homomorphisms are
categorically equivalent to Specker MV-algebras,  as well as  to
unital Specker $\ell$-groups. 
The deep connections between  searching games with errors/lies and
their underlying combinatorial-algebraic structures 
might be of help to enrich  the theory of adaptive error correcting codes 
(see \cite[and references therein]{cic} for a detailed account) 
with combinatorial game-theoretical  tools from 
category theory.

\section{Preliminary material}
\label{section:preliminary}

We refer to \cite{bkw} for $\ell$-groups,
to \cite{ada,bor} for category theory,
and  to \cite{cigmun, cigdotmun, mun-jfa} for
MV-algebras and the functors  $\Gamma$
and $\Xi$.

 By an {\it $\ell$-group} we mean a lattice-ordered abelian group.
A  (distinguished, strong, order)  {\it unit}  of an
$\ell$-group $G$ is an element  $u\geq 0$
with the property that for every $g\in G$ there is $n\in 
 \Z_{>0}=\{1,2,\dots\}$
such that    $- nu \leq  g \leq  nu.$
The pair  $(G,u)$ is said to be a {\it unital} $\ell$-group.
 Unless otherwise specified, 
 every ideal of  $(G,u)$ in this paper 
 (also known as an $\ell$-ideal) 
is {\it proper},
i.e., it is different from $G$. 
A {\it unital $\ell$-homomorphism} between two
unital $\ell$-groups  $(G,u_G)$ and $(H,u_H)$
is an $\ell$-homomorphism  $\psi\colon G\to H$
such that $\psi(u_G)=u_H$.

\begin{definition}[\cite{bez, bkw, con, condar,  spe}]
  A {\it singular} element  of  $G$ is an element
$0\leq s\in G$ such that,
for all $0\leq a \leq s$,\,\,\,\,\,
$a \wedge (s-a)=0$.\footnote{In \cite[Definition 11.2.7]{bkw}
one finds the additional assumption  $s>0$.  
This will have no effect
on the results of this paper.}
A {\it Specker} $\ell$-group $S$ is an $\ell$-group that
is generated, as a group, by its singular elements.
 A {\it unital Specker $\ell$-group} is a pair $(S,u)$ 
 where $S$ is a Specker $\ell$-group 
 and $u$ is a  unit of $S$. 
\end{definition}

By  \cite[Lemma 2.11]{bez},
weak and strong units coincide in Specker $\ell$-groups.
Any unital Specker $\ell$-group  
 $(S,u)$ has a greatest singular element,   denoted
$s_S$.
This immediately follows, e.g., from \cite[Lemma 2.11 and Corollary 2.12]{bez}. 
 We let
$$
 \uSlg
$$
denote the full subcategory of unital $\ell$-groups
given by unital Specker $\ell$-groups and unital 
$\ell$-hom\-om\-orph\-isms.

 For any unital $\ell$-group $(G,u)$,
  the space  $\maxspec(G)$ of maximal ideals of
$(G,u)$ is equipped with  
 the maximal spectral topology inherited by restriction of the
 spectral topology from prime to maximal ideals.
 In particular, by definition, a subbasis (in fact, 
  a basis) of open sets in $\maxspec(G)$ is given by the
 family of sets of the form
 $$
 \{\mathfrak m \in \maxspec(G)
  \mid g\notin \mathfrak m\},
 $$
 letting $g$ range over all elements of $G$.
 A basis of closed sets is then given by the
 family of sets of the form
 $$
 F_g=\{\mathfrak m \in \maxspec(G) \mid g \in \mathfrak m\},
 $$
 for all $g \in G$.
See \cite[Definition 10.1.2 and \S 10.2]{bkw}.
The closed sets in $\maxspec(G)$ are precisely the sets of the form
$\{\mathfrak m\in \maxspec(G)\mid
\mathfrak j \subseteq \mathfrak m\}$,
for $\mathfrak j$ a (possibly improper) ideal of $G$.
See  \cite[Proposition 10.1.7]{bkw}.\footnote{In \cref{corollary:misto}, 
we will observe that the assignment $\mathfrak j \mapsto
 \{\mathfrak m\in \maxspec(G)\mid \mathfrak j \subseteq \mathfrak m\}$ 
 from the set of possibly improper ideals of $(G,u)$ to the set of 
 closed subsets of $\maxspec(G)$ is a bijection (in fact, an order 
 anti-isomorphism), with inverse $C \mapsto \bigcap C$.}

\medskip

\begin{lemma} 
\label{l:unique-iso}
For any  totally ordered nontrivial Specker $\ell$-group $S$   there
	 is a unique $\ell$-isomorphism of $S$ onto $\Z$, once the latter 
is  equipped with the natural order.
\end{lemma}

\begin{proof}
	This immediately 
	follows from the fact that any two nonzero singular elements of $S$ coincide.
	As a matter of fact,  let $a$ and $b$ be nonzero singular elements.
	Since $S$ is totally ordered, either $a \leq b$ or $b \leq a$.
	Without loss of generality,   $a \leq b$.
	Since $b$ is singular, $a \land (b - a) = 0$.
	Since $S$ is totally ordered, either $a = 0$ or $b - a = 0$.
	Since $a$ is nonzero, $b - a = 0$, i.e., $a = b$.
\end{proof}

\begin{lemma}
 \label{l:maps}
	Any surjective $\ell$-homomorphism between $\ell$-groups maps singular elements to singular elements.
	Therefore, homomorphic images (i.e., quotients) of Specker $\ell$-groups are Specker $\ell$-groups.
\end{lemma}

\begin{proof}
	Let $f \colon S \twoheadrightarrow T$ be a surjective $\ell$-homomorphism between $\ell$-groups.
	Let $s \in S$ be a singular element.
	We show that $f(s)$ is a singular element.
From $s \geq 0$ we get  $f(s) \geq 0$.
For any $b \in T$  such that $0 \leq b \leq f(s)$  let us prove $b \land (f(s) - b) = 0$.
	By surjectivity, there is $a \in S$ such that $f(a) = b$.
	Set $a' = (a \lor 0) \land s$.
	Then, $0 \leq a' \leq s$, and hence, since $s$ is a singular element, we have $a' \land (s - a') = 0$.
	Therefore, $f(a') \land (f(s) - f(a')) = 0$.
	Moreover, $f(a') = f((a \lor 0) \land s) = (f(a) \lor 0) \land f(s) = f(a)$, whence
 $f(a) \land (f(s) - f(a)) = 0$.
	This proves that $f$ maps singular elements to singular elements. The rest is immediate.
\end{proof}

\begin{theorem}
 \label{theorem:unique-rho}
	Let  $S$ be a Specker $\ell$-group
	and  $\mathfrak m$ a maximal ideal of $S$.
	Then there is a unique  $\ell$-homomorphism $\rho_{\m}$ of $S$ 
	onto $\Z$ whose {\em kernel}  (i.e., the preimage of $\{0\}$) is $\m$.
\end{theorem}

\begin{proof}
	Since any maximal ideal of $S$ is prime  
(\cite[Corollary 9.2.4]{bkw}),
	$S/\mathfrak m$ is a nontrivial totally ordered
	$\ell$-group.
	By \cref{l:maps}, $S/\mathfrak m$ is a Specker $\ell$-group.
By \cref{l:unique-iso}, there is a unique $\ell$-isomorphism $\iota_\mathfrak{m} \colon S/\mathfrak{m} \to \Z$.
\end{proof}

\begin{lemma}
 \label{lemma:compute-rho}
	Let $(S, u)$ be a unital Specker $\ell$-group, $\m$ a maximal ideal, $\rho_\m \colon S \to \Z$ the unique surjective $\ell$-homomorphism whose kernel is $\m$, and $s_S$ the greatest singular element of $S$.
	We then have:
	\begin{enumerate}[label = (\roman*)]
	
		\item
		$ \rho_\mathfrak{m}(s_S) = 1$.
	
		\item For every $g \in S$, $\rho_\m(g)$ is the unique $j \in \Z$ 
		such that $g - j  s_S \in \mathfrak{m}$.
	
	\end{enumerate}
\end{lemma}

\begin{proof}
	(i) By \cref{l:maps}, $\rho_\m \colon S \to \Z$ preserves singular elements.
	Then any singular element of $S$ is mapped either to $0$ or to $1$.
	Since the set of singular elements of $S$ generates $S$ and $\rho_\m$ is surjective, the image under $\rho_\m$ of the set of singular elements of $S$ generates $\Z$.
	Therefore, there is a singular element of $S$ whose image under $\rho_\m$ is $1$.
Since $\rho_\m$ is order-preserving, the image under $\rho_\m$
 of the greatest singular element of $S$ is $1$.
	
	(ii)
	This immediately follows from (i).
\end{proof}

\medskip
 Throughout, the set  $\Z$ is equipped 
 with the discrete topology.
For any   boolean space $X$ and 
continuous function $u\colon X\to \Z_{>0}$  we will write 
$$(C_X, u)$$
to denote the unital
$\ell$-group $C_X$ of all continuous $\Z$-valued
functions on $X$,  equipped with the  unit 
$u$.
From the compactness of  $X$  and the discreteness of the topology
of  $\Z$ it follows that the range of $u$ (and in general of any continuous function from $X$ to $\Z_{>0}$) is finite.

\begin{notation} 
\label{n:natural}
	Let $(S,u)$ be a unital Specker $\ell$-group.
	For any $g\in S$  we define the map 
	\begin{align*}
		g_S^\natural
		\colon  \maxspec(S) &\longrightarrow
		\Z\\
		\m & \longmapsto \rho_{\m}(g),
	\end{align*} 
	where $\rho_\m$ is the unique surjective $\ell$-homomorphism from $S$ to $\Z$ with kernel $\m$,
	as given by  \cref{theorem:unique-rho}.
	In other words (by \cref{lemma:compute-rho}),
	\begin{equation}
	\label{equation:natural}
	\mbox{
		$g_S^\natural (\mathfrak m) = $   
		the unique  \,\,\,$ j \in \Z$ \,\,\,such that 
		$g - j  s_S  \in  \mathfrak m$,}
	\end{equation}
	where $s_S$ is the greatest singular element of $S$.
\end{notation}

\begin{lemma} \label{lemma:natural-is-0}
	Let $(S,u)$ be a unital Specker $\ell$-group. For any $g \in S$ and any $\m \in \maxspec(S)$, we have $g^{\natural}(\m) = 0$ if and only if $g \in \m$. 
\end{lemma}

\begin{proof}
	The value 
	$g^{\natural}(\m) = \rho_\m(g)$
	 is $0$ if and only if $g \in \ker(\rho_\m) = \m$.  
\end{proof}

\begin{theorem}
	\label{theorem:accozzaglia}
	Let $(S,u)$ be a unital Specker $\ell$-group and $s_S$
	its greatest singular
	element.  We then have:
	\begin{enumerate}[label = (\roman*)]
	
		\item
		For any $\mathfrak n\in \maxspec(S)$ there is 
		a unique unital $\ell$-isomorphism
		of   the maximal quotient $(S,u)/\mathfrak n$ onto the unital
		$\ell$-group $(\Z, u^\natural_S(\mathfrak{n}))$.
		
		\item
		For any $g\in S$,  the map $ g_S^\natural
		\colon  \maxspec(S)\to\Z$ is continuous, i.e.\
		$g^\natural_S \in C_{\maxspec(S)}$.
		
		\item
		 $\maxspec(S)$ is a boolean space,
		and  the map\,\,\, $^\natural\colon
		g\in (S,u)\mapsto g_S^\natural\in \Z^{\maxspec(S)}$
		is a  unital $\ell$-isomorphism
		of  $(S,u)$ onto  $(C_{\maxspec(S)},u^\natural)$.\footnote{For a different proof  see
		\cite[Proposition 2.7(ii)]{mun-sl}, which makes use of  
		\cite[Lemma 2.6]{mun-sl}. In this latter lemma
		the obviously necessary assumption that $C_X$
		is separating (i.e., that the compact space $X$
		 is a boolean space) was inadvertently omitted. This has no effect on
		the results of that paper. }
	
		\item
		For any boolean space $X$  the singular elements of $C_X$
		 are precisely the characteristic functions of clopen subsets of $X$.
		Consequently, the set \,\,$\mathsf{sing}(S)$\,\,\,  of  singular elements of  $(S, u)$,
		with the
		lattice structure inherited from 
		$S$ by restriction,  is
		a  boolean algebra with top element  $s_S$ and
		bottom element the zero element of $S$. 
		This boolean algebra is
		isomorphic to the algebra
		of clopen sets of $\maxspec(S)$. 
		An isomorphism is given by the map 
		$$\mathsf{supp}\colon s\in \mathsf{sing}(S)
		\mapsto \{\mathfrak m\in \maxspec(S)\mid s/\mathfrak m>0\} = \{\mathfrak m\in \maxspec(S)\mid s \notin \mathfrak m\}.$$

		\item
		Every boolean space  is homeomorphic to the
		maximal spectral space of some   unital Specker $\ell$-group.
		
	\end{enumerate}
\end{theorem}

\begin{proof}
	(i)   Immediate from  Lemma \ref{lemma:compute-rho}.
	
	\smallskip
	(ii)    Pick a basic open set  $C\subseteq \Z$, 
	say $C=\{c\},\,\,c\in \Z$,
	with the intent of showing that  $(g_S^\natural)^{-1}(C)$
	is an open subset of  $\maxspec(S)$.
	As a matter of fact, 
	\begin{align*}
		(g_S^\natural)^{-1}(C)&=\{\mathfrak m\in
		\maxspec(S)\mid g_S^\natural(\mathfrak m)\in C\}
		= \{\mathfrak m\in \maxspec(S)\mid g_S^\natural(\mathfrak m) = c\}\\
		&= \{\mathfrak m\in \maxspec(S)\mid {g-c s_S}\in {\mathfrak m}\}.
	\end{align*}
	Since $(S,u)$ is a unital Specker $\ell$-group, 
	every element $l\in S$ can be written as a linear combination of 
	finitely many singular elements  $s_1,\dots,s_n\in S$ 
	with integer coefficients $k_1,\dots, k_n$.
	For any $\m \in \maxspec(S)$ and  $i \in \{1, \dots, n\}$
	 the element ${s_i}^\natural_S(\m)$ is a singular element of $\Z$, and hence it is either 
	$0$ or $1$; this is so because 
	${s_i}^\natural_S(\m)$
	 is the image of $s_i$ under the surjective 
	$\ell$-homomorphism $\rho_\m$ of \cref{theorem:unique-rho},
	which preserves singular elements by \cref{l:maps}.
	Therefore, the range of the map
	$l_S^\natural\colon \maxspec(S) \to\Z $
	 is contained in $\{\sum_{i=1}^n k_ie_i \mid (e_i)_{i=1}^n \in \{0,1\}^n\}$, which
	is a finite subset of $\Z$. 
	This holds in the particular case when $l= g-c s_S$.
	By  definition of  $\maxspec(S)$,
	the set $ \{\mathfrak m\in \maxspec(S)\mid {g
		-c s_S}\in {\mathfrak m}\}$ is a (subbasic) closed subset of  $\maxspec(S)$.
	It is also an open subset of $\maxspec(S)$, 
	because, as we have just seen,
	the map 
	$(g-c s_S)^\natural\colon \maxspec(S) \to\Z $
	has only finitely many   values, and each of
	them is  attained over a closed subset of  $\maxspec(S)$.
	Hence,  each fiber of $(g-c s_S)^\natural $  is clopen. 
	We conclude that 
	$(g^\natural_S)^{-1}(C)=\{\mathfrak m\in
	\maxspec(S)\mid g^\natural_S(\mathfrak m)=c\}$
	is  an  open set of $\maxspec(S)$, as desired to prove that  
	$g^\natural_S$ is continuous.

	\smallskip

	(iii)  The boolean space $X$ in the representation 
	\cite[Corollary 2.12]{bez}  of
	any  unital Specker $\ell$-group  $(G,u)$  
	is homeomorphic 
	to the maximal spectral space of   $(G,u)$. 
	This follows, e.g., from \cite[Theorem 4.16(iv)]{mun11}.
	Surjectivity follows from \cite[Corollary 2.12]{bez}, or 
	\cite[Theorem 13.5.3]{bkw}.

	\smallskip
	(iv) 
	That the singular elements of $C_X$ are precisely 
	the characteristic functions of the clopen subsets of $X$ is straightforward.
	By (iii), $(S, u)$ is isomorphic to $(C_{\maxspec(S)},u^\natural)$.
	The singular elements of $C_{\maxspec(S)}$ 
	are the characteristic functions of the clopen sets of $\maxspec(S)$.
	By \cref{lemma:natural-is-0}, for every $g \in S$ the support of the function $g^\natural \in C_{\maxspec(S)}$ is $\{\m \in \maxspec(S) \mid g \notin \m\}$.
		Therefore, for every singular element $s \in S$, $s^\natural$ is the characteristic function of $\{\m \in \maxspec(S) \mid s \notin \m\}$.
	(Compare with  \cite[Theorem 13.5.3]{bkw}.)

	\smallskip
	
	(v)  Let $X$ be a boolean space. Let  $(C_X, 1)$ be the unital
	Specker $\ell$-group of all continuous integer-valued functions on $X$, with the
	constant function 1 over $X$ as the distinguished unit of $C_X$. 
	Since this is also the greatest singular element of $C_X$, from the
	main result of
	\cite{mun-tac} it follows that the  $\Gamma$ functor  transforms  $(C_X, 1)$
	into the boolean algebra $B$ consisting of all continuous
	$\{0,1\}$-valued functions on $X$. 
	Both $(C_X, 1)$ and  $B$ are separating.
	As a particular case of  \cite[Theorem 4.16(iv)]{mun11}, $X$ is
	homeomorphic to $ \maxspec(B)$. 
	By   \cite[Theorem 7.2.2]{cigdotmun}, 
	the preservation properties of $\Gamma$ now yield a homeomorphism
	$ \maxspec(C_X)\cong \maxspec(B)$.
	In conclusion, we have the desired homeomorphisms
	$
	X\cong \maxspec(C_X)\cong \maxspec(B).
	$
\end{proof}

\begin{corollary}[{\cite[Proposition 3.4]{con}}] 
 \label{corollary:hyperarchimedean}
   Every unital Specker
 $\ell$-group $(S,u)$ is  {\em hyper\-ar\-chim\-edean}, in the sense that
  any prime ideal
 of $(S,u)$ is maximal.
 \end{corollary}
 
 \begin{proof}
 In view of   Theorem \ref{theorem:accozzaglia}(iii),
 it suffices to argue assuming that $(S,u)$ has the form
 $(C_X,u_X)$ for some boolean multispace $(X, u_X)$.
 Since every map $h\in C_X$ is integer-valued and has  a finite range, for 
any  $0\leq f,g\in C_X$ there is $n\in \Z_{>0}$  such that
$$nf\wedge g = (n+1)f\wedge g.$$
  By  \cite[Theorem 14.1.2 (i)$\Leftrightarrow$(vi)]{bkw},
this condition is equivalent to saying that  $S$ is hyperarchimedean.
\end{proof}

\begin{corollary}
\label{corollary:misto}
Let   $(S,u)$ be a unital Specker
 $\ell$-group.
 Every  ideal of $S$ is the intersection of all maximal
ideals of $S$ containing it.\footnote{This is true also for the improper ideal, as long as one uses the convention that the intersection of an empty family of subsets of $S$ is $S$.}
\end{corollary}

\begin{proof}
 This follows from 
  Corollary \ref{corollary:hyperarchimedean},  since
in any  $\ell$-group 
 every ideal is the intersection of all prime ideals containing it.
\end{proof}

\section{Boolean multispaces and the functor \texorpdfstring{$\SS$}{S}}

\begin{definition}
\label{definition:bms}
 The {\em category $\mathsf{Bms}$ of  
 boolean multispaces}  is defined by:

\medskip
\noindent
$\mathsf{Bms}$-$\mathsf{OBJECT}$:  
A pair  $(X,u_X)$ with  $X$ a boolean space 
and $u_X$ a  continuous function from $X$ to
$\Z_{>0}= \{1,2,\dots\}$.

For $x \in X$, $u(x)$ is called the \emph{multiplicity} of $x$. The function $u$ is called the \emph{multiplicity (function)} on $X$.

\medskip
\noindent
$\mathsf{Bms}$-$\mathsf{MORPHISM}$ from a
$\mathsf{Bms}$-object $(X,u_X)$ to a $\mathsf{Bms}$-object
$(Y,u_Y)$: a  continuous
function $\gamma\colon X\to Y$ such that 
the function $\zeta_\gamma$ defined by  
$$\zeta_\gamma = \frac{u_X}{u_Y(\gamma)}$$
is $\Z_{>0}$-valued.
In other words,
 for all $x \in X$, the positive integer $u_Y(\gamma(x))$ 
 divides $u_X(x)$ (i.e., $\gamma$ decreases multiplicity 
 with respect to the divisibility order of $\Z_{>0}$).
We say that $\zeta_\gamma$ is the {\em multiplicity} of
$\gamma$.
Note that $\zeta_\gamma$ is continuous.

\medskip 
The composition of $\mathsf{Bms}$-morphisms
$\eta_1\colon (X,u_X)\to (Y,u_Y)$ and
$\eta_2\colon (Y,u_Y)\to (Z,u_Z)$  is
the morphism
$\eta_2(\eta_1)\colon (X,u_X)\to (Z,u_Z)$.

\smallskip
We recall that every continuous function between compact 
Hausdorff spaces (in particular, every $\Bms$-morphism)
 is a closed map, meaning that it maps closed sets to closed sets
 \cite[Theorem 3.1.12]{eng}.
Therefore, a  $\mathsf{Bms}$-isomorphism from 
$(X,u_X)$ to $(Y,u_Y)$ amounts to a continuous bijection $\eta \colon X \to Y$ such that
$u_Y(\eta(x))=u_X(x)$ for all $x\in X$.
\end{definition}

\begin{definition}
	\label{definition:bms-to-uslg}\hfill
	\begin{enumerate}[label = (\roman*)]
	
		\item The functor $\mathcal S\colon \mathsf{Bms}\to
		  \uSlg^{\mathrm{op}}$ sends
		  every  $\mathsf{Bms}$-object $(X,u)$ to
		  the $ \uSlg$-object 
		  $(C_X, u)$:
		\begin{equation}
		\label{equation:S-on-objects}
		\mathcal S\colon (X,u)\mapsto (C_X,u),
		\,\,\,\text{ for any $(X,u)\in \mathsf{Bms}$.} 
		\end{equation}

		\item Furthermore,  $\mathcal S$ transforms any
	    $\mathsf{Bms}$-morphism
		$
		\gamma\colon (W,u_W)\to (V,u_V)
		$ 
		into the unital $\ell$-homomorphism    
		$
		\mathcal S_\gamma\colon \mathcal S(V,u_V)
		\to  \mathcal S(W,u_W)
		$
		given  by
		\begin{equation}
		\label{equation:esse-gamma}
		\mathcal S_\gamma\colon f\in C_V\mapsto 
		f(\gamma) \frac{u_W}{u_V(\gamma)} \in C_W,
		\,\,\,\mbox{ for any $f\in C_V$.} 
		\end{equation}
	
	\end{enumerate}

\end{definition}

This is indeed a well-defined assignment on morphisms:

\begin{proposition}
\label{proposition:routine}
Adopt the above notation.

	\begin{enumerate}[label = (\roman*)]
	
		\item
		\medskip $\mathcal S_\gamma(u_V)= 
	u_V(\gamma)  \frac{u_W}{u_V(\gamma)} = u_W$,
	whence $\mathcal S_\gamma$ preserves units.
	
	\medskip
	
		\item
		$\mathcal S_\gamma(0)=  0$.
	
	\medskip
		\item $\mathcal S_\gamma(f+g)=\mathcal S_\gamma(f)+\mathcal S_\gamma(g)$,
	and the like for $\vee, 	\wedge$ and $-$, which shows that
	 $\mathcal S_\gamma$ is a unital $\ell$-homomorphism of 
	 $(C_V,u_V)$ into $(C_W, u_W)$. \qed
	
	\end{enumerate}

 \end{proposition}
 
 \medskip

  \begin{proposition}
 \label{proposition:essentially-surjective}
  The functor 
   $\mathcal S$ is {\em  essentially surjective}, in the sense that 
 for any  $(S,u)\in \uSlg^{\mathrm{op}}$ there is  
 $(X,v)\in \mathsf{Bms}$ such that
 $\mathcal S(X,v)\cong (S,u)$.
 \end{proposition}

 \begin{proof}
 By  Theorem \ref{theorem:accozzaglia}(iii)
 we have $(\maxspec(S),u^\natural) \in  \mathsf{Bms}$ and 
 $(S,u) \cong \mathcal{S}(\maxspec(S),u^\natural)$,
 whence  $(\maxspec(S),u^\natural)$ is the desired $ \mathsf{Bms}$-object.
 \end{proof}

 \medskip

 \begin{proposition}
 \label{proposition:faithful} 
 \,\,\,The functor $\mathcal S$ is {\em faithful}:
 For any two distinct $\mathsf{Bms}$-morphisms  
 $\gamma \not=\gamma' \colon
  (W,u_W)\to (V,u_V)$,
 the unital $\ell$-hom\-o\-mor\-ph\-isms $\mathcal S_{\gamma}$ and
  $\mathcal S_{\gamma'}$ are distinct. 
\end{proposition}

\begin{proof}  $\gamma\not=\gamma'$ means that there is
$y\in W$ with $\gamma(y)\not=\gamma'(y)$. 
Since $\gamma(y)$ and $\gamma'(y)$ are distinct and $C_V$
is  separating, there is a map  $f\in C_V$ with $f(\gamma(y))=0$ and $f(\gamma'(y))\not=0$.
Now
$$
(\mathcal S_\gamma(f))(y)= 
 f(\gamma(y))  \frac{u_W(y)}{u_V(\gamma(y))}  =0,\,\,\,
\mbox{ but }\,\,\,
(\mathcal S_{\gamma'}(f))(y)= 
f(\gamma'(y))  \frac{u_W(y)}{u_V(\gamma'(y))} \not=0,
$$
which shows that $\mathcal S_\gamma\not=\mathcal S_{\gamma'}$.
\end{proof}

\begin{lemma} \hfill
\label{lemma:three-maps}

\smallskip
	\begin{enumerate}[label = (\roman*)]
	
		\item
		For  $X$  a fixed but otherwise arbitrary  boolean space, 
 let  $(C_X,u)$ be the
unital Specker $\ell$-group of all   $\Z$-valued
continuous functions on $X$, with the distinguished order unit $u$.
Letting 
 $Z_Xf=\{x\in X\mid f(x)=0\}$, the map   $\dot Z_X\colon  \maxspec(C_X) \to X$  given by
  $$
  \mathfrak m\in \maxspec(C_X)\mapsto 
\dot Z_X\mathfrak m = x_\mathfrak m= \mbox{\,the unique element of }
\bigcap\{Z_Xf\mid f\in \mathfrak m\}
$$
is a homeomorphism of $\maxspec(C_X)$ onto $X$. 
   Its inverse  $\mathfrak M_X$  maps  any
$x\in X$ to the maximal ideal $\mathfrak M_X(x) = 
\mathfrak m_x$ of $C_X$ given by all functions of
$C_X$ vanishing at $x$. 

\medskip
	\item
	For every $x \in X$ and $g \in C_X$ we have the identity
 $g(x) = g^\natural(\mathfrak{M}_X(x))$.
Thus the functions $\mathfrak{M}_X$ and $\dot Z_X$
yield  $\Bms$-isomorphisms between $(X, u)$ and $(\maxspec(C_X), u^\natural)$. 
	\end{enumerate} 
\end{lemma}

\begin{proof} 
(i) 
For every $x \in X$  one easily sees that $\mathfrak{M}_X(x)$
 is  a maximal ideal of $C_X$.
Moreover,
we have the identity
\begin{equation} \label{equation:intersection}
	\bigcap\{Z_X f \mid f \in \mathfrak{M}_X(x)\} = \{x\}.
\end{equation}
For the  $\subseteq$-inclusion, 
 since $X$ is a boolean space, for every $y \in X$ different from 
 $x$ there is $f \in C_X$ such that $f(x) = 0$ and $f(y) \neq 0$.
 The $\supseteq$-inclusion is trivial.

\medskip
We now proceed as follows:

\medskip
\noindent{\it Claim:}
For any $\mathfrak m \in \maxspec(C_X)$ and  
$x \in \bigcap\{Z_X f \mid f \in \mathfrak m\}$  we have 
	\begin{equation} \label{eq:m-is-M}
		\mathfrak{m} = \mathfrak{M}_X(x)
	\end{equation}
	and
	\begin{equation} \label{equation:intersection-is}
		\bigcap \{Z_X f \mid f \in \mathfrak{m}\} = \{x\}.
	\end{equation}

\medskip	
Indeed, by  hypothesis 
$\mathfrak{m} \subseteq \mathfrak{M}_X(x)$, 
whence  \eqref{eq:m-is-M} follows from $\mathfrak{m}$ being maximal.  Then 
\eqref{equation:intersection-is} follows from \eqref{eq:m-is-M} and \eqref{equation:intersection}.
Our claim is thus settled.

\smallskip 
 
Next, we must  prove that $\dot Z_X$ is well-defined. To this purpose,
 let $\mathfrak{m} \in \maxspec(C_X)$, 
with the intent of proving  that the set $\bigcap \{Z_X f \mid f \in \mathfrak{m}\}$ is a singleton.
By  our claim, 
$\bigcap \{Z_X f \mid f \in \mathfrak{m}\}$ has at most one element.
Let us prove that it has at least one element.
For every finite subset $S \subseteq \mathfrak {m}$, we have
\[
\bigcap\{Z_X f \mid f \in S\} = Z_X \left(\sum_{f \in S}\lvert f \rvert\right),
\]
which is nonempty because $\sum_{f \in S}\lvert f \rvert \in \mathfrak{m}$ and $\mathfrak{m}$ is proper.
The  compactness of $X$ now yields  $\bigcap\{Z_X f \mid f \in \mathfrak m\} \neq \emptyset$.
In conclusion, $\bigcap\{Z_X f \mid f \in \mathfrak m\}$ is a singleton, and 
  $\dot Z_X$ is well-defined.
  
 \smallskip

From \eqref{equation:intersection} 
it follows that, for all $x \in X$, \,\, $x = \dot{Z}_X\mathfrak{M}_X(x)$.
On the  other hand, by
 our claim,  for all $\mathfrak{m} \in \maxspec(C_X)$, 
 \,\,  $\mathfrak{m} = \mathfrak{M}_X\dot{Z}_X \mathfrak{m}$.  
 We have  thus shown  that 
 $\dot{Z}_X$ and $\mathfrak{M}_X$ are mutually inverse functions, yielding
  a  {\it bijection}
between $\maxspec(C_X)$ and $X$.

\medskip
There remains to prove that $\mathfrak M_X$
 is a homeomorphism. Since $\mathfrak{M}_X$ is a bijection between compact 
  Hausdorff spaces, it is enough to prove that $\mathfrak M_X$
 is continuous.
Let  $C$ be a subbasic closed set in $\maxspec(C_X)$,
 say
$C= F_f = $  the  set of all maximal ideals of  $C_X$  to which 
$f$ belongs, where $f$ is a fixed but otherwise
arbitrary element of $C_X$. Then
\begin{align*}
	\mathfrak M_X^{-1}(C) & = \{x \in X \mid \mathfrak M_X(x) \in C\} \\
	& = \{x \in X \mid \{g \in C_X \mid g(x) = 0\} \in F_f\}\\
	& = \{x \in X \mid f \in \{g \in C_X \mid g(x) = 0\}\}\\
	& = \{x \in X \mid f(x) = 0\}\\
	& = Z_X f,
\end{align*}
 a closed set by the continuity of $f$. This shows that 
 $\mathfrak M_X$
 is a homeomorphism. 

\medskip
(ii)
For any  $x \in X$ and $g \in C_X$ we will  prove $g(x) = g^\natural(\mathfrak{M}_X(x))$.
The evaluation function
\begin{align*}
	\mathrm{ev}_x \colon C_X & \longrightarrow \Z\\
	f & \longmapsto f(x)
\end{align*}
is  surjective and its  kernel coincides with  $\mathfrak{M}_X(x)$.
Thus the evaluation function
coincides with  the function $\rho_{\mathfrak{M}_X(x)}$ 
 in \cref{theorem:unique-rho} and in the definition of $^\natural$ (\cref{n:natural}).
Therefore, for every $g \in C_X$, 
$$g^\natural(\mathfrak{M}_X(x)) 
= \mathrm{ev}_x(g)= g(x),
$$
 as desired.
In  particular,   $u(x) = u^\natural(\mathfrak{M}_X(x))$, whence both
 $\mathfrak{M}_X$ and its inverse $\dot Z_X$ are $\Bms$-isomorphisms.
 \end{proof}
 
In what follows we will need the result stating  that the preimage under a unital $\ell$-homomorphism of a maximal ideal
of a unital $\ell$-group is maximal.
While this is folklore, we will give a self-contained proof in Corollary
 \ref{corollary:pullback-of-maximal}. As a preliminary step,
 in the following lemma
 we will give an elementary characterization of maximal ideals.
For readers familiar with commutative algebra, we mention
 that this characterization can be seen as an analogue 
 in our setting of the following characterization of maximal ideals
  in a commutative ring $R$ with unit: an ideal $\mathfrak{j}$
   is maximal if and only if $1 \notin \mathfrak{j}$ and for all 
   $a \in R \setminus \mathfrak{j}$ there is $b \in R$ such that 
   $1 - ba \in \mathfrak{j}$ (i.e., $R/\mathfrak{j}$ is a field) \cite[p.~3]{atimac}.

\begin{lemma}\label{lemma:characterization-maximal}
	Let $(G, u)$ be a unital $\ell$-group and $\mathfrak{j}$ 
	a (possibly improper) ideal of $G$.
	The following conditions are equivalent:\footnote{It is well known 
	that these conditions are also equivalent to the existence 
	of an injective unital $\ell$-homomorphism from 
	$(G/\mathfrak{m}, u/\mathfrak{m})$ to $(\mathbb{R}, 1)$. 
	Moreover, if $G$ is a Specker $\ell$-group, these conditions 
	are also equivalent to $G/\mathfrak{j}$ being isomorphic 
	to $\Z$ as an $\ell$-group (the nontrivial direction being \cref{theorem:unique-rho}).}
	\begin{enumerate}[label = (\roman*)]
	
		\item
		$\mathfrak j$ is maximal.
		
		\item
		$u \notin \mathfrak j$ (i.e.,\ $\mathfrak j$ is proper) and for all 
		$a \in G \setminus \mathfrak j$ there is 
	$n \in \Z_{\geq 0}$ such that $(u - n  \lvert a \rvert) \lor 0 \in \mathfrak j$ (i.e., $n  \lvert a \rvert /\mathfrak{j} \geq u / \mathfrak{j}$).	
	
	\end{enumerate}
\end{lemma}

\begin{proof}
	The ideal $\mathfrak{j}$ is maximal if and only if it is proper (a condition that is easily seen to be equivalent to $u \notin \mathfrak j$) and, for all $a \in G \setminus \mathfrak j$, the ideal generated by $\mathfrak{j} \cup \{a\}$ is $G$ (i.e., it contains $u$). The ideal generated by $\mathfrak{j} \cup \{a\}$ contains $u$ if and only if there are $n \in \Z_{\geq 0}$ and $j \in \mathfrak{j}$ such that $u \leq j + n  \lvert a \rvert$, which 
	is the case  if and only if there is $n \in \Z_{\geq 0}$ such that $u - n  \lvert a \rvert$ is below some element of $\mathfrak{j}$, which is equivalent to $(u - n  \lvert a \rvert) \lor 0 \in \mathfrak j$.
\end{proof}

While for commutative rings with unit it is not true that the preimage of a maximal ideal is a maximal ideal (as for the inclusion $\Z \hookrightarrow \mathbb{Q}$), for unital $\ell$-groups we have:
 
\begin{corollary}[Preimage of maximal is maximal]
 \label{corollary:pullback-of-maximal}
	Let $\psi \colon  (G, u) \to (H,v)$ be a unital $\ell$-homomorphism between unital $\ell$-groups.
	For every maximal ideal $\m$ of $H$
	the set $\psi^{-1}(\m) = \{g \in G \mid \psi(g) \in \m\}$ is a maximal ideal of $G$.
\end{corollary}

\begin{proof}
	This is an immediate consequence of \cref{lemma:characterization-maximal}.
	(For a proof  in the context of MV-algebras
	see \cite[Proposition 1.2.2]{cigdotmun}.
	For an alternative  proof,
		combine  \cite[Proposition 1.2.16(i)]{cigdotmun} with  \cite[Corollary 7.2.3]{cigdotmun}.)
\end{proof}
 
\begin{lemma}
 \label{lemma:B-on-morphisms}
	Let $\psi\colon  (S, u) \to (T,v)$ be a unital $\ell$-homomorphism between unital Specker $\ell$-groups.
	The map  
	\begin{align*}
		\psi^{-1} \colon (\maxspec(T), v^\natural) & \longrightarrow (\maxspec(S), u^\natural)\\
		\mathfrak{m} & \longmapsto \psi^{-1}(\m) = \{f\in S\mid \psi(f) \in \mathfrak m\}
	\end{align*}
	is continuous.
	Moreover, for any $\mathfrak{m} \in \maxspec(T)$ 
	 there is $k \in \Z_{>0}$ such that for all $g \in S$ 
	 we have the identity  $(\psi(g))^\natural(\m) = k  g^{\natural}(\psi^{-1}(\m))$.
	In  particular, $v^\natural(\m) = k  u^{\natural}(\psi^{-1}(\m))$, whence
  $\psi^{-1}$ is a $\Bms$-morphism.
\end{lemma}

\begin{proof}
	The map is well-defined because, by \cref{corollary:pullback-of-maximal}, for any  $\mathfrak m\in\maxspec(T)$ the set
	$\psi^{-1}(\mathfrak m)=\{f\in S\mid \psi(f)\in \mathfrak m\}$ 
	is a maximal ideal of $S$.
	
	\medskip
	We next prove the continuity of $\psi^{-1}$.
	We recall that a subbasis for the closed sets of $\maxspec(S)$ is given by the sets of the form
	$$
	F_f=\{\mathfrak n \in \maxspec(S)\mid
	f\in \mathfrak n\},
	$$
	for $f$  ranging over arbitrary
	elements of  $S$ (and analogously for $T$, mutatis mutandis).
	To prove that $\psi^{-1} \colon \maxspec(T) \to \maxspec(S)$ is continuous it is enough to prove that the preimage under the function $\psi^{-1}$ of any element of this subbasis for the closed sets of $\maxspec(S)$ is a closed subset of $\maxspec(T)$.
	Let $f \in S$. We have
	\begin{align*}
		(\psi^{-1})^{-1}(F_f)&= \{\mathfrak{m} \in \maxspec(T) \mid \psi^{-1}(\mathfrak{m}) \in F_f\} \\
		&= \{\mathfrak m\in \maxspec(T)\mid
		f\in  \psi^{-1}(\mathfrak m) \}\\
		&= \{\mathfrak m\in \maxspec(T)\mid \psi(f)\in  \mathfrak m\},
	\end{align*}
	a closed subset of $\maxspec(T)$.
	This concludes the proof of 
	the continuity of the map $ \psi^{-1} \colon \maxspec(T) \to \maxspec(S)$.
	
	\medskip
	There remains to prove
	\begin{equation}
	\label{equation:multiplicities}
 \psi^{-1} \,\, \mbox{decreases multiplicities}.
	\end{equation}

	Let $\mathfrak{m} \in \maxspec(T)$.
	By  Theorem \ref{theorem:unique-rho}, there is a unique surjective 
	$\ell$-homomorphism $\rho_\m \colon T \to \Z$ such that 
	$\ker \rho_\m = \mathfrak{m}$, and there is a unique surjective 
	$\ell$-homomorphism $\rho_{\psi^{-1}(\m)} \colon S \to \Z$ 
	such that $\ker \rho_{\psi^{-1}(\m)} = \psi^{-1}(\m)$.
	Since $\psi^{-1}(\m)$ is the kernel of the surjective 
	$\ell$-homomorphism $\rho_{\psi^{-1}(\m)} \colon S \to \Z$ 
	and is also the kernel of the composite 
	$S \xrightarrow{\psi} T \xrightarrow{\rho_\m} \Z$ (which 
	need not be surjective), there is a unique 
	$\ell$-homomorphism $\gamma \colon \Z \to \Z$ 
	such that the following diagram commutes:
	\[
	\begin{tikzcd}
		S \arrow{r}{\psi} \arrow[swap]{d}{\rho_{\psi^{-1}(\m)}}& T \arrow{d}{\rho_\m}\\
		\Z \arrow[dashed,swap]{r}{\gamma}& \Z
	\end{tikzcd}
	\]
	Any $\ell$-homomorphism from $\Z$ to $\Z$ 
	amounts to
	 multiplication by a number in $\Z_{\geq 0}$.  
	 Therefore, there is $k \in \Z_{\geq 0}$
	  such that, for every $j \in \Z$, $\gamma(j) = k  j$.
 
	It remains to show that $k\not=0$.
	Since surjective $\ell$-homomorphisms preserve (strong order)  units, $\rho_{\m}(v)$
	 is a (strong order) unit of $\Z$, and hence it is nonzero.
	We have
	\begin{align*}
		k  \rho_{\psi^{-1}(\m)}(u) 
		= \gamma(\rho_{\psi^{-1}(\m)}(u)) = \rho_\m(\psi(u)) = \rho_\m(v) \neq 0.
	\end{align*}
	Therefore, $k \neq 0$, as desired. 
	This settles \eqref{equation:multiplicities} and completes the proof.
\end{proof}

\begin{corollary}[of \,\cref{lemma:three-maps,lemma:B-on-morphisms}]
\label{corollary:psi-dual}
For any boolean multispaces
$(W,u_W)$ and $(V, u_V)$, and\,
$ \uSlg$-morphism
$$
\psi\colon (C_V,u_V)\to (C_W,u_W),
$$
let \,\,$\psi_{\rm dual}\colon W\to V$\,\, be 
  the composite function 
  $\dot Z_V\psi^{-1}\mathfrak M_W$. In more detail,
\begin{align*}
\psi_{\rm dual} \colon y\in W \,\,\,&\longmapsto_{\mathfrak M_W}&& \hspace{-6.7em}
\mathfrak m_y\in \maxspec(C_W)\\
{\,\,\,}&\longmapsto &&\hspace{-6.7em}\mathfrak \psi^{-1}(\mathfrak m_y) 
\in \maxspec(C_V)\\
\,\,\,&\longmapsto_{\dot Z_V} &&
\hspace{-6.7em}x= \dot Z_V\psi^{-1}\mathfrak M_W(y)	\in V.
\end{align*}
Then \,\,$\psi_{\rm dual}$ is  a $\Bms$-morphism
of  $(W, u_W)$  into  $(V, u_V)$.
\hfill{$\Box$}
 \end{corollary}

 \bigskip

\section{\texorpdfstring{$\SS$}{S} is a categorical equivalence}
\label{section:full}

 \begin{proposition} [$\mathcal S$ is full]
 \label{proposition:full}
 For any boolean multispaces $(W,u_W)$ and $(V, u_V)$ and 
unital $\ell$-hom\-omor\-phism
$$\psi\colon \mathcal S(V, u_V)=(C_V,u_V)
\to  \mathcal S(W, u_W)=(C_W,u_W)$$
there exists a (unique by the faithfulness of $\mathcal S$, \cref{proposition:faithful})
$\mathsf{Bms}$-morphism
 $$
 \gamma \colon (W,u_W) \to (V,u_V)
 $$
such that $\psi=\mathcal S_\gamma$.
Specifically,
$\gamma$  coincides with the map    
  $\psi_{\rm dual}$.
\end{proposition}

\begin{proof}	
As per \cref{corollary:psi-dual},  the $\Bms$-morphism $\psi_{\rm dual}$
is the composition of the $\Bms$-morphisms
	\[
	(W,u_W) \xrightarrow{\mathfrak{M}_W} (\maxspec (C_W),u_W^\natural)
	 \xrightarrow{\psi^{-1}}  (\maxspec (C_V), u_V^\natural) \xrightarrow{\dot Z_V} (V, u_V).
	\]
To get  the identity $\SS_{\psi_{\rm dual}} = \psi$
we must  prove 
$$
\mbox{ for all $f \in C_V$ and $x \in W$, $\SS_{\psi_{\rm dual}}(f)(x) = \psi(f)(x)$.
}
$$
To this purpose,  let us arbitrarily fix $f \in C_V$ and $x \in W$. 
	By the definition of the functor $\SS$ on morphisms we can write
	\[
	\SS_{\psi_{\rm dual}}(f)(x) = f(\psi_{\rm dual}(x))  \frac{u_W(x)}{u_V(\psi_{\rm dual}(x))}.
	\]
 There remains to prove
	\begin{equation} 
	\label{equation:tbs}
		\psi(f)(x) = f(\psi_{\rm dual}(x))  \frac{u_W(x)}{u_V(\psi_{\rm dual}(x))}.
	\end{equation}
By  definition of $\psi_{\rm dual}$, the element $\psi_{\rm dual}(x)$ 
is the unique $y \in V$ such that for any $h \in C_V$, 
\,\,\, $\psi(h)(x) = 0$\, implies \, $h(y) = 0$.
	Since $u_W(x) > 0$, 
	$$
	\mbox{$\frac{\psi(f)(x)}{u_W(x)}$ is a well-defined rational number, say 
	${p}/{q}$, with $p \in \Z$ and $q \in \Z_{>0}$.}
	$$
We then  have the identities
	\[
	\psi(q  f - p  u_V)(x) = q  \psi(f)(x) - p  \psi(u_V)(x) = q  \psi(f)(x) - p  u_W(x)= 0.
	\]
As a consequence, 
	\[
		0 = (q  f - p  u_V)(\psi_{\rm dual}(x)) = q  f(\psi_{\rm dual}(x)) - p  u_V(\psi_{\rm dual}(x)).
	\]
We have just proved 
	\[
	f(\psi_{\rm dual}(x)) = \frac{p}{q}  u_V(\psi_{\rm dual}(x)) 
	= \frac{\psi(f)(x)}{u_W(x)}  u_V(\psi_{\rm dual}(x)),
	\]
	which amounts to 
  \eqref{equation:tbs}.
\end{proof}

\begin{theorem}
\label{theorem:duality}
 The functor $\mathcal S$
 is an equivalence 
between  the category
 $\mathsf{Bms}$ of boolean multispaces and 
 the opposite  $ \uSlg^{\rm op}$
  of the category
 of unital Specker $\ell$-groups.  
\end{theorem}

\begin{proof} As a joint effect of Propositions
\ref{proposition:essentially-surjective}, 
\ref{proposition:faithful},
\ref{proposition:full}.
\end{proof}

\begin{corollary}
\label{corollary:stone}
  Stone duality is the particular case of  Theorem
  \ref{theorem:duality}
  when  $\mathcal S$ is  restricted to
  boolean multispaces with constant multiplicity $1$, and
   $ \uSlg$ is restricted to
 the category of  
  unital  Specker $\ell$-groups whose distinguished unit is
  singular.
 \end{corollary}

 \begin{proof}
 In {\rm  \cite{mun-tac}} 
 it was proved that, by restriction,  the functor
   $\Gamma$ of \cite[\S 3]{mun-jfa}  yields a categorical equivalence
  between unital Specker $\ell$-groups whose distinguished unit is
  singular,  and boolean algebras.
 By Stone duality,
the category $\mathsf{BA}$  of  boolean algebras
and their homomorphisms  is dual to the full subcategory 
$\mathsf{Bms1}$ of
 boolean multispacs of the form  $(X, \boldsymbol 1)$
  with $\boldsymbol 1$ the constant function $1$ over 
the boolean space $X$. Trivially, 
  $\mathsf{Bms1}$ is categorically equivalent to
$\mathsf{BA}$.
\end{proof}

\begin{definition}
\label{definition:specker-mv-algebra}
A  {\em Specker MV-algebra} is  an MV-algebra
isomorphic to 
 $\Gamma(S,u)$ for some   unital Specker $\ell$-group
 $(S,u)$. We denote by
 $$
 \mathsf{SMV}
 $$
 the category of Specker MV-algebras with
 MV-algebraic homomorphisms.
\end{definition}

\begin{lemma}[Compare with {\cite[Theorem 4.5]{mun-sl}}] 
\label{lemma:mv-representation}
Up to isomorphism, any Specker MV-algebra  $A$  is a finite product
of MV-algebras of the form $C(X_i, \mbox{\L}_{n(i)})$, 
where each \,$X_i$\,
is a (nonempty) boolean space and\, \L$_{n(i)}$\, is the \luk\ chain with \,$n(i)+1$ elements.
Equivalently,  $A$ is isomorphic to a finite product of tensor products
$B_i\otimes $\L$_{n(i)}$, for  boolean algebras  $B_i$.
\end{lemma}

\begin{proof}
By Theorem \ref{theorem:accozzaglia}(iii) we
  may identify  $(S,u)$ with $(C_X, v)$,  for some boolean space
$X$, where $v\colon X\to \Z_{>0}$.  The continuous map
$v$  splits $(C_X, v)$ into a finite product 
$$
(C_X, v)\cong (C_{X_1},v_1)\times\dots \times (C_{X_l},v_l)
$$
of unital
Specker $\ell$-groups ($C_{X_i},v_i)$, where $v_i$ takes
a constant integer value  (also denoted $v_i>0$) over each boolean space
$X_i$.  For short,  $X_1,\dots, X_l$
are the {\em fibers} of  $v$.
Direct inspection shows that $\Gamma$ transforms
each ($C_{X_i},v_i)$ into the MV-algebra of all continuous
functions on $X_i$  taking values in  \L$_{v_i}$.
Now note that   $\Gamma$ preserves finite products. 
For the second statement see \cite[Corollary 4.6]{mun-sl}.
 \end{proof}
 
In \cite[Cor.~4.5]{fusgehgoomar}, the authors obtained an ``extended Priestley'' duality for MV-algebras; in this duality, the dual of an MV-algebra is the Priestley space dual to the underlying bounded distributive lattice of the MV-algebra, equipped with additional structure.
The structures dual to MV-algebras according to this duality are called \emph{MV-spaces} \cite[Def.~4.4]{fusgehgoomar}.
 
 \begin{proposition}
  \label{proposition:priestley}
	The category $\mathsf{Bms}$ of boolean multispaces is equivalent to the full subcategory of the category of MV-spaces consisting of those MV-spaces that correspond to Specker MV-algebras under the extended Priestley duality for MV-algebras in \cite[Cor.~4.5]{fusgehgoomar}.
 \end{proposition}

\begin{proof}
	This follows from the duality between MV-algebras and MV-spaces in \cite[Cor.~4.5]{fusgehgoomar} and the duality between boolean multispaces and unital Specker $\ell$-groups in \cref{theorem:duality}.
\end{proof}
 
\section{The functor \texorpdfstring{$\SS$}{S} as a part of a duality}
\label{section:adjoint}

\begin{definition}[{\cite[Definitions 1.5.4, 1.4.1]{rie}}]
\label{definition:equivalence}
An {\em equivalence
of categories}   consists of functors 
$F \colon \mathsf{C}  \leftrightarrows  \mathsf{D} \cocolon G $ together
with {\it natural isomorphisms} $$\eta \colon  1_ \mathsf{C} \cong
 GF, \,\,\,
\epsilon \colon  FG 
\cong 1_{ \mathsf{D}}.$$
Two  functors
$F \colon  \mathsf{C} \leftrightarrows  \mathsf{D} \cocolon G $ 
 are said to be
 {\it part of an equivalence}  of 
 $ \mathsf{C}$ and $ \mathsf{D}$ 
   if there are natural isomorphisms
$\eta$ and $\epsilon$  such that the quartet $(F, G, \eta, \epsilon)$ 
 constitutes  an equivalence  of 
 $ \mathsf{C}$ and $ \mathsf{D}$ in the above sense.
 \end{definition}

Thus in category theory the same term ``equivalence''
applies  both to a functor $F$ and to a quartet
$(F, G, \eta, \epsilon)$ with $F,G$ functors and
$\epsilon, \eta$ natural isomorphisms.

\medskip
We next introduce a functor which, in Theorem 
\ref{theorem:enter-B-tris}, will be shown to 
be a part  with $\mathcal S$ of an equivalence
between the categories $ \mathsf{Bms}$ and $ \uSlg^{\rm op}$.

  \begin{definition}
 \label{definition:B}\hfill
 \begin{enumerate}[label = (\roman*)]
 
 	\item
 The functor  
 $\mathcal B\colon  \uSlg^{\rm op} \to  \Bms$
 transforms every unital Specker $\ell$-group
$(S,u)$ into the boolean multispace
$$
\mathcal B(S,u)=(\maxspec(S), u^\natural).
$$
As per \cref{n:natural}, the function
$u^\natural\colon  \maxspec(S) \to \Z$
maps any $\mathfrak m\in \maxspec(S)$ into the value 
at $u$ of the unique surjective $\ell$-homomorphism from 
$S$ to $\Z$ with kernel $\m$; equivalently, into the unique
integer  $j = j_{u,\mathfrak m}$ such that 
$u -j s_S \in  \mathfrak m$.

 \smallskip
 	\item
	 On any unital
 $\ell$-homomorphism 
 $\psi \colon (S, u) \to (T, v)$
 of unital Specker $\ell$-groups, the functor 
   $\mathcal B$ \, sends $\psi$  
to the morphism $\B_\psi \colon \maxspec(T)
\to \maxspec(S)$ that maps any 
 $\mathfrak{m}\in \maxspec(T)$ to the maximal
 ideal of $S$ given by  $\psi^{-1}(\mathfrak{m})
 =\{f\in S\mid \psi(f) \in \mathfrak m\}$, as per 
 Lemma \ref{lemma:B-on-morphisms}.
 
 \end{enumerate}
 
 \end{definition}

\smallskip
 
A standard fact in category theory is that, for every full,
 faithful and essentially surjective functor $F \colon \mathsf{C}
  \to \mathsf{D}$ there is an equivalence $(F, G, \eta, \epsilon)$
   with $F$ as the first component.
For example, this is proved in  the $(\Leftarrow)$-direction of 
\cite[Theorem~1.5.9]{rie}.

The proof runs as follows: 

\noindent
From the essential surjectivity of $F$ and the axiom of choice for classes 
one gets  a family $(G_Y)_{Y \in \mathsf{D}}$ of objects of 
$\mathsf{C}$
 and a family $(\epsilon_Y \colon F(G_Y) \to Y)_{Y \in \mathsf{D}}$
 of isomorphisms in $D$.
Using  $F$  together with these two families, and the fact 
that $F$ is full and faithful,  one  constructs  an equivalence $(F, G, \eta, \epsilon)$.
In case we are  already given a family $(G_Y)_{Y \in \mathsf{D}}$ 
of objects of $\mathsf{C}$ and a family $(\epsilon_Y \colon F(G_Y) \to Y)_{Y \in \mathsf{D}}$ 
of isomorphisms in $D$, one can dispense with both the essential surjectivity
assumption for $F$ and the axiom of choice (still keeping the
 faithfulness and fullness assumptions), 
and jump directly to the point of the proof where one constructs the
desired  equivalence with these ingredients.

\smallskip
In more detail we have:

\begin{lemma}
\label{lemma:first-diagram} 
Let $F \colon \mathsf{C} \to \mathsf{D}$ 
be a full and faithful functor.
For any $Y \in \mathsf{D}$  let $G_Y$ be an object of
$\mathsf{C}$ and $\epsilon_Y \colon F(G_Y) \to Y$ an isomorphism in $\mathsf{D}$.
Then $(F, G, \eta, \epsilon)$ is a well-defined {\em adjoint equivalence}
of categories\footnote{An \emph{adjoint equivalence of categories}
is an equivalence of categories $(F, G, \eta, \epsilon)$ that is also
an adjunction; in other words, an adjunction where both the unit 
and the counit are natural isomorphisms.  If two functors $F$ and $G$ 
are part of an equivalence $(F, G, \eta, \epsilon)$, they are also part of an adjoint equivalence.
It suffices to appropriately modify either one of  $\eta$ or  $\epsilon$.
 See, e.g.,  \cite[\S IV.4]{mac} or \cite[Proposition 4.4.5]{rie}.}
   \begin{itemize}

	\smallskip
  	\item
  	$G \colon \mathsf{D} \to \mathsf{C}$ is the  functor
	given by the following stipulations:
  	
  	---to any object 
  		$Y \in \mathsf{D}$ assign the object $G_Y$;
  		
  	---to any  morphism $g \colon Y \to Y'$ in $\mathsf{D}$ 
  		assign the (existing and unique, by
		the  fullness and the faithfulness of $F$) morphism $f \colon G_Y \to G_{Y'}$ in 
  		$\mathsf{C}$ making the following diagram commute:
  		\begin{equation} \label{eq:a-label-to-cite-this-equation-counit}
  			\begin{tikzcd}
  				F(G_Y) \arrow[swap]{d}{F(f)} \arrow{r}{\epsilon_Y}& Y  \arrow{d}{g} \\
  				F(G_{Y'}) \arrow[swap]{r}{\epsilon_{Y'}} & Y'\,;
  			\end{tikzcd}
  		\end{equation}

	\smallskip  
	  \item 
  	$(\epsilon_Y \colon F(G_Y) \to Y)_{Y \in \mathsf{D}}$ is the given collection of isomorphisms, which is a natural transformation from $FG$ to $\mathrm{id}_{\mathsf{D}}$;
 
 	\smallskip 	
  	\item
  	$(\eta_X \colon X \to G_{F(X)})_{X \in \mathsf{C}}$
	is the natural transformation from $\mathrm{id}_{\mathsf{C}}$ to $GF$ whose component $\eta_X$ at $X$ is the
	 (existing and unique, by the fullness and the faithfulness of $F$) 
	 morphism such that $F(\eta_X) \colon F(X) \to F(G_{F(X)})$ 
	 is the inverse of $\epsilon_{F(X)} \colon F(G_{F(X)}) \to F(X)$.
  	
  \end{itemize}
\end{lemma}

\begin{proof}
To see that $(F, G, \eta, \epsilon)$ is an equivalence, one can just follow the  
$(\Leftarrow)$-direction of  the proof of 
\cite[Theorem~1.5.9]{rie}, 
starting from the line  that reads:

\begin{quote}
 ``For each $\ell \colon d \to d'$, Lemma 1.5.10 defines a unique morphism
making the square [...]''.
\end{quote}
As observed in the proof of \cite[Theorem IV.4 (iii)$\Rightarrow$(ii)]{mac},
this equivalence is indeed an adjoint equivalence, because it satisfies the two triangle identities:
\smallskip

---The triangle identity $\epsilon_{F(X)} \circ F (\eta_X)\,\,\,\, (=
\epsilon_{F(X)}(F (\eta_X) )) \,\,=\,\, 1_{F(X)}$,
which  holds by the definition of $\eta$;

	\smallskip
---The triangle identity $G(\epsilon_Y) \circ \eta_{G(Y)} = 1_{G(Y)}$.

\noindent
To prove this latter identity,
  since $F$ is a full and faithful functor it is enough to prove  
\begin{equation}
\label{equation:identity} 
FG(\epsilon_Y) \circ F(\eta_{G(Y)}) = 1_{FG(Y)}.
\end{equation}
Applying to $\epsilon_{Y}$ the definition of $G$ on morphisms
we obtain the  following commutative diagram:
\begin{equation*}
	\begin{tikzcd}
		FGFG(Y) \arrow[swap]{d}{FG(\epsilon_Y)} 
		\arrow{r}{\epsilon_{FG(Y)}}& FG(Y)  \arrow{d}{\epsilon_Y} \\
		FG(Y) \arrow[swap]{r}{\epsilon_{Y}} & Y\,.
	\end{tikzcd}
\end{equation*}
Since $\epsilon_Y$ is an isomorphism,  $FG(\epsilon_Y) = \epsilon_{FG(Y)}$.
The  identity  $\epsilon_{FG(Y)} \circ F(\eta_{G(Y)}) = 1_{FG(Y)}$
 follows from the definition of $\eta_{G(Y)}$.  
 
 Having thus obtained the
 identity  \eqref{equation:identity}, the proof
 is  complete.
	\end{proof}

\noindent
As we aim to establish a {\it duality}, 
we will need the following  variant 
of Lemma \ref{lemma:first-diagram}:

\begin{lemma}
	\label{lemma:contravariant}
	Let $F \colon \mathsf{C} \to \mathsf{D}^\mathrm{op}$ 
	be a full and faithful functor.
	For each $Y \in \mathsf{D}$ let $G_Y$ be an object of 
	$\mathsf{C}$ and $\epsilon_Y \colon Y \to F(G_Y)$ 
	an isomorphism in $\mathsf{D}$.
	Then  $(F, G, \eta, \epsilon)$ is an adjoint dual equivalence, 
	where: 
	\begin{itemize}

		\smallskip
		\item
		$G \colon \mathsf{D}^\mathrm{op} \to \mathsf{C}$ is the  functor
		given by the following stipulations:
		
		\smallskip
		---to any object 
		$Y \in \mathsf{D}$ assign the object $G_Y$;
		
		---to any  morphism $g \colon Y \to Y'$ in $\mathsf{D}$ 
		assign the uniquely determined morphism $f \colon G_{Y'} \to G_{Y}$
		in $\mathsf{C}$ making the following diagram commute:
		\begin{equation}
			\label{eq:a-label-to-cite-this-equation-counit-dual}
			\begin{tikzcd}
				Y  \arrow[swap]{d}{g} \arrow{r}{\epsilon_Y} & F(G_Y) \arrow{d}{F(f)} \\
				Y' \arrow[swap]{r}{\epsilon_{Y'}} & F(G_{Y'}). 
			\end{tikzcd}
		\end{equation}

		\smallskip  
		\item 
		$(\epsilon_Y \colon Y \to F(G_Y))_{Y \in \mathsf{D}}$ is the given collection of isomorphisms, which is a natural transformation from $\mathrm{id}_{\mathsf{D}}$ to $FG$;
		
		\smallskip 	
		\item
		$(\eta_X \colon X \to G_{F(X)})_{X \in \mathsf{C}}$
		  is the natural transformation from $\mathrm{id}_{\mathsf{C}}$ to $GF$ whose component $\eta_X$ at $X$ is the unique
		morphism such that $F(\eta_X) \colon F(G_{F(X)}) \to F(X)$ 
		is the inverse of $\epsilon_{F(X)} \colon F(X) \to F(G_{F(X)})$. \qed
		
	\end{itemize}
\end{lemma}
 
 \medskip
 
\begin{lemma} 
\label{lemma:description-of-the-unique}
	\,\,For any morphism \,\,\,$\psi \colon (S, u) \to (T, v)$\,\,\, in\, $\uSlg$, the
	(necessarily unique)   morphism 
	$\lambda \colon (\maxspec(T), v^\natural) \to (\maxspec(S), u^\natural)$
	of Lemma \ref{lemma:contravariant}  making the following diagram commute
	\begin{equation} \label{eq:naturality-of-natural-simple}
		\begin{tikzcd}
			(S,u) \arrow{r}{^\natural} \arrow[swap]{d}{\psi} 
			& (C_{\maxspec(S)}, u^\natural) \arrow{d}{\SS_\lambda}\\
			(T, v) \arrow[swap]{r}{^\natural} 
			& (C_{\maxspec(T)}, v^\natural)
		\end{tikzcd}
	\end{equation}
	maps any $\mathfrak{m}\in \maxspec(T)$ \,to the set\, $\psi^{-1}(\mathfrak{m}) 
	= \{f\in S\mid \psi(f) \in \mathfrak m\}$.
\end{lemma}

\begin{proof}
	By Theorem  \ref{theorem:accozzaglia}(iii), for every $(Q,w)$ in $\uSlg$ the map 
	$$\natural \colon(Q,w) \to (C_{\maxspec(Q)}, w^\natural)$$
	 is an isomorphism in $\uSlg$.
	Therefore, the commutativity of \eqref{eq:naturality-of-natural-simple} 
	is equivalent to writing $\SS_\lambda = \natural \circ \psi \circ \natural^{-1}$.
	Since, by  Propositions \ref{proposition:faithful} and \ref{proposition:full},
	 $\SS$ is full and faithful,
	there is a unique morphism $\lambda \colon (\maxspec(T), v^\natural) \to (\maxspec(S), 
	u^\natural)$ making \eqref{eq:naturality-of-natural-simple} commute.
	Let $\lambda$ be such a morphism.
	Let $g \in S$,  with the intent of
	proving  $g \in  \lambda(\mathfrak{m}) \iff g \in \psi^{-1}(\mathfrak{m})$.
	We then  have
	\begin{align*}
		(\psi(g)^\natural)(\mathfrak{m}) & = (\SS_\lambda(g^\natural))(\mathfrak{m}) && \text{by the commutativity of diagram \eqref{eq:naturality-of-natural-simple}}\\
		& = g^\natural(\lambda(\mathfrak m))
		 \frac{v^\natural(\mathfrak m)}{u^\natural(\lambda(\mathfrak m))} && \text{by Definition \ref{definition:bms-to-uslg}.}
	\end{align*}

	 \medskip
	
	\noindent
This latter equality will find use in the third line in the following chain of equivalences:
\begin{align*}
	g \in \psi^{-1}(\m)   & \iff \psi(g) \in \m\\
	& \iff (\psi(g)^\natural)(\mathfrak{m}) = 0\, \, \, \text{by \cref{lemma:natural-is-0}}\\
	& \iff g^\natural(\lambda(\mathfrak m)) = 0 \,\,\, \text{since $\frac{v^\natural(\mathfrak m)}{u^\natural(\lambda(\mathfrak m))}\neq 0$}\\
	& \iff g \in \lambda(\m) \, \, \, \text{by \cref{lemma:natural-is-0}}.
\end{align*}
This completes the proof.
\end{proof}

\bigskip
\noindent
Applying this result to the
 full and faithful
  functor $\mathcal S\colon \mathsf{Bms} 
\to  \uSlg^{\rm op}$ along with 
 the collection of $\uSlg$-isomorphisms $(^\natural \colon
  (S,u) \to \SS(\maxspec(S), u^\natural))_{(S,u) \in \uSlg}$
(Theorem \ref{theorem:accozzaglia}(iii)),
we obtain:

\smallskip

\begin{theorem}
\label{theorem:enter-B-tris}
The functors $\SS$ and $\B$ are part of an adjoint dual equivalence between the category $\Bms$ of boolean multispaces and the category $\uSlg$ of unital Specker $\ell$-groups, with the following unit and counit:

	\begin{enumerate}[label = (\roman*)]
	
		\item
		$(\mathfrak{M}_X \colon (X, u) \to (\maxspec(C_{X}),u^\natural))_{(X,u) \in \Bms}$\,;

	\smallskip	
	
	\item
	$(^\natural \colon (S, u) \to (C_{\maxspec(S)}, u^\natural))_{(S, u) \in \uSlg}$.
	
	\end{enumerate}
	
\end{theorem}

\begin{proof}
We will use Lemma \ref{lemma:contravariant} replacing
\begin{itemize}
\item the category  $\mathsf{C}$ by $\Bms$, 
\item the category $\mathsf{D}$ by $\uSlg$, 
\item the functor $F$ by $\SS$, 
\item the assignment $Y \in \mathsf{D} \mapsto G_Y \in \mathsf{C}$ by
 the assignment $(S,u) \in \uSlg \mapsto (\maxspec(S), u^\natural) \in \Bms$, 
 and 
 \item
 the assignment $Y \mapsto (\epsilon_Y \colon Y \to F(G_Y))$ by the assignment  
	$(S, u) \mapsto (\natural \colon (S, u) \to \SS(\maxspec(S), u^\natural))$.
 \end{itemize}

 \noindent
By Propositions  \ref{proposition:faithful} and \ref{proposition:full},  
 $\SS$ is faithful and full.
By  Theorem \ref{theorem:accozzaglia}(iii),
	 for every object $(S,u) \in \uSlg$ the function 
	 $^\natural \colon (S, u) \to \SS(\maxspec(S), u^\natural)$ is an isomorphism.
In view of   \cref{lemma:contravariant}, 
 a functor 
	$$\B' \colon \uSlg^\mathrm{op} \to \Bms$$
	 such that $\SS$ and $\B'$ are part of a duality is obtainable by
	 the following stipulations:

\smallskip	
	---let $\B'$ send any object $(S,u) \in \uSlg$ to the object $(\maxspec(S), u^\natural)$, and

\smallskip	
	---let $\B'$ send any morphism $\psi \colon (S, u) \to (T, v)$ in $\uSlg$ 
	to  the {\it uniquely determined}  (as specified below)  morphism $\maxspec(T)
	\to \maxspec(S)$ making the following diagram commute:
	\begin{equation*}
		\begin{tikzcd}
			(S,u) \arrow{r}{^\natural} \arrow[swap]{d}{\psi} 
			& \SS(\maxspec(S), u^\natural) \arrow{d}{\SS_\lambda}\\
			(T, v) \arrow[swap]{r}{^\natural} 
			& \SS(\maxspec(T), v^\natural).
		\end{tikzcd}
	\end{equation*}

\noindent	
	By \cref{lemma:description-of-the-unique},
	for  any $\mathfrak{m}\in \maxspec(T)$
	 this morphism  necessarily transforms the maximal ideal  $\mathfrak m$
	 into the maximal ideal   $\psi^{-1}(\mathfrak{m}) = \{f\in S\mid \psi(f) \in \mathfrak m\}$.
	We conclude that  $\B'$ (is well defined and)  {\it coincides with the functor}  $\B$
	of \cref{definition:B}. 
	Therefore,  $\B$  is  part of a duality with  $\SS$.
	
	Finally, to prove that the unit is $(\mathfrak{M}_X \colon (X, u) \to (\maxspec(C_{X}),u^\natural))_{(X,u) \in \Bms}$, we must
	 prove that, for every $(X, u) \in \Bms$, the function $\SS(\mathfrak{M}_X)
	  \colon \SS\B\SS(X,u) \to \SS(X,u)$ is the inverse of 
	  $^\natural \colon \SS(X,u) \to \SS\B\SS(X,u)$.
	Since these are isomorphisms, it is enough to prove that
	the following composite is the identity:
	\[
		\SS(X,u)\xrightarrow{^\natural}\SS\B\SS(X,u) \xrightarrow{\SS(\mathfrak{M}_X)}\SS(X,u). 
	\]
	The function $\SS(\mathfrak{M}_X) \colon C_{\maxspec(C_X)} \to C_X$,
	obtained by applying $\SS$ to the morphism $\mathfrak{M}_X 
	\colon (X,u) \to (\maxspec(C_X), u^\natural)$, maps any $g \in 
	C_{\maxspec(C_X)}$ to the function $\SS(\mathfrak{M}_X)(g) \in
	 C_X$
	   that maps $x \in X$ to $g(\mathfrak{M}_X(x))
  \frac{u(x)}{u^\natural(\mathfrak{M}_X(x))}$. 
	  The latter product 
  equals $g(\mathfrak{M}_X(x))$ because, 
	  by \cref{lemma:three-maps}(ii), $u^\natural(\mathfrak{M}_X(x)) = u(x)$.
	In particular, for every $f \in C_X$, \,\,
	$\SS(\mathfrak{M}_X)$ maps $f^\natural$ to the function
	 that maps $x$ to $f^\natural(\mathfrak{M}_X(x))$. The latter 
	 equals $f(x)$ by \cref{lemma:three-maps}(ii).
	Therefore, for every $f \in C_X$, $\SS(\mathfrak{M}_X)(f^\natural) = f$.
We have shown that the composite
		\[
	\SS(X,u)\xrightarrow{^\natural}\SS\B\SS(X,u) \xrightarrow{\SS(\mathfrak{M}_X)}\SS(X,u) 
	\]
	is the identity, and hence
	 the unit is $(\mathfrak{M}_X \colon (X, u)
	  \to (\maxspec(C_{X}),u^\natural))_{(X,u) \in \Bms}$.
The proof is complete.
\end{proof}

 \medskip

\section{Limits in \texorpdfstring{$\Bms$}{Bms} and colimits in the categories \texorpdfstring{$\uSlg$}{uSlg} and \texorpdfstring{$\SMV$}{SMV}}
	\label{section:sei}

In this section  we consider the existence problem for limits in $\Bms$, 
and  for  colimits in $\uSlg$
and  $\mathsf{SMV}$.

We refer to \cite[Definition 2.4.1]{bor} and \cite[\S 12]{ada} for background on limits and colimits.
For products in category theory we refer to \cite[Definition 10.19]{ada}
and  \cite[Definition 2.1.1]{bor}.
For coproducts see  \cite[Table 10.63]{ada}  and \cite[Definition 2.2.1]{bor}.

We show that $ \mathsf{Bms}$ has all finite limits
 (i.e., limits of diagrams indexed by a category with a finite set of objects and a finite set of morphisms).
All these limits have an easy description, as they are computed 
as in the category $\Set$ of sets and functions. Moreover,
 they are computed as in the category of boolean spaces.
It will follow that the dual categories $\uSlg$, and $\mathsf{SMV}$ have finite colimits. 
 
For a finite subset $I \subseteq \Z_{>0}$ 
 we let $\LCM(I)$ denote the least common multiple of $I$, i.e., the supremum of $I$ in the poset $\Z_{>0}$ ordered by divisibility.

\begin{theorem}[Description of finite limits]
 \label{t:finite-limits}
	\,\,\,The category $\Bms$ of boolean multispaces is finitely complete.
	More generally, $\Bms$ has limits of all diagrams indexed by a category with a finite set of objects.
	
	If $D \colon \mathsf{I} \to \Bms$ is one such diagram, then a limit $(p_i \colon (L,v) \to D(i))_{i \in \mathsf{I}}$ over $D$ is given by:

	\begin{enumerate}[label=(\roman*)]
		\item \label{i:the-limit}
		$L = \{(x_i)_{i \in \mathsf{I}} \in \prod_{i\in I}D(i)\mid \text{for all }\mathsf{I}\text{-morphism }h \colon i \to j,\, D(h)(x_i) = x_j\}$;

\smallskip
		\item \label{i:the-maps}
		for each $i \in \mathsf{I}$, $p_i \colon L \to D(i)$ maps $(x_j)_{j \in \mathsf{I}}$ to $x_i$; 

\smallskip		
		\item \label{i:the-topology}
		the topology on $L$ is the coarsest topology making every $p_i$ continuous; equivalently, the subspace topology induced by the product topology of $\prod_{i\in I}D(i)$;

\smallskip		
		\item \label{i:the-multiplicities}
		for $x \in L$, $v(x) = \LCM\{u_i(p_i(x)) \mid i \in \mathsf{I}\}$,
		 where $u_i$ is the multiplicity of $D(i)$.
	\end{enumerate}
	
\end{theorem}

\begin{proof}
	The proof is straightforward. Items 
	(i--ii) amount to the usual construction of limits in $\Set$.
	Likewise, 
	(i--iii) amount to the usual construction of limits
	 in the category of (topological spaces, as well as in its full subcategory of) boolean spaces.
	The finiteness of the set of objects of \,\,$\mathsf{I}$\,\,
	 is used to show that the function
	\begin{align*}
		(\Z_{>0})^{\mathrm{Ob}(\mathsf{I})} & \longrightarrow \Z_{>0}\\
		(k_i)_{i \in \mathsf{I}} & \longmapsto \LCM\{k_i \mid i \in \mathsf{I}\}
	\end{align*}
	is well-defined and continuous,  because 
	$(\Z_{>0})^{\mathrm{Ob}(\mathsf{I})}$ is discrete, 
	being a finite power of discrete spaces.  This guarantees 
	that $v$ is well-defined and continuous.
\end{proof}

\begin{remark}
	Note that the limits in \cref{t:finite-limits} (in particular, all finite limits) 
	are computed in the same way as they are in $\Set$,  in the category 
	of topological spaces, and in the category of boolean spaces.
\end{remark}

\begin{remark}
From \cref{t:finite-limits} we immediately have:
	\begin{quote}
	{\it The category $\Bms$ of boolean multispaces
	 has all finite products, equalizers and pullbacks} 
	 (which are all special instances of finite limits).
	 \end{quote}
Thus for instance:
	\begin{enumerate}[label = (\roman*)]
		\item
		The product 
		$(X,u)\times (X',u')$ in $\Bms$
		is given by  
		$(X\times X', \,\, \LCM(u, u'))$,
		where the function 
		$ \LCM(u, u')\colon  X\times X' \to \Z_{>0}$
		maps  any
		pair  $(x,x')\in X\times X'$  to the least
		common multiple of  $u(x)$ and $u'(x')$.
		The (canonical projection) morphism of
		$(X,u)\times (X',u')$  into $(X,u)$ (resp., into $(X',u')$)  is
		determined by  the projection
		of  $X\times X'$  onto $X$  (resp., onto $X'$).
		
		\item
		The category  $\mathsf{Bms}$  of boolean multispaces 
		has a (unique up to isomorphism)
		terminal object $(X,u)$, namely the singleton $X=\{*\}$
		equipped with the function $u$ defined by $u(*)=1$.

		\item
		Let $f,g\colon (X,u) \rightrightarrows (X',u')$  be 
		two parallel  $\mathsf{Bms}$-morphisms.
		The equalizer of $f$ and $g$ is $(E,e)$, where
		$E=\{x\in X\mid f(x)=g(x)\}$ and $e\colon E\to \Z_{>0}$  coincides with
		$u$ over $E$, together with the inclusion $(E,e) \hookrightarrow (X,u)$.
	\end{enumerate}
	Furthermore, by \cite[Theorem 12.4 (1)$\Rightarrow$(3)]{ada},
	 from the finite completeness of $\Bms$ it follows that 
	$\Bms$ 
	 {\it has finite intersections}\footnote{A category is said to \emph{have finite 
	 intersections} if for every object $X$ the partially ordered class 
	 of subobjects of $X$ has finite infima. See, e.g., \cite[Definition~12.1.(4)]{ada}.}.
\end{remark}

\begin{corollary} 
\label{corollary:finitely-cocomplete}
	The categories of unital Specker $\ell$-groups and 
	Specker MV-algebras are finitely cocomplete. 
	More generally, both categories have a colimit for every 
	diagram indexed by a category with a finite set of objects.
\end{corollary}

We have just proved that the category $\Bms$ of boolean multispaces is finitely complete.
We now prove that it is not complete, as it lacks some infinite products.

\begin{proposition}
 \label{proposition:no-countable-power}
	The category $\Bms$ is not complete. For example, there is no countable power of the boolean multispace $(\{a,b\}, u)$ with $u(a) = 1$ and\,\, $u(b) = 2$.
\end{proposition}

\begin{proof}
We first sketch the main idea of the proof.
	Roughly speaking, the product of countably many copies of $(\{a,b\}, u)$ does not exist,
	 because the natural choice for the multiplicity function on the cartesian 
	 product $\{a,b\}^\omega$---namely the function 
	 $(x_i)_{i} \mapsto \LCM\{u(x_i) \mid i \in \omega\}$---is not continuous.
	This is essentially due to the fact that the function 
	$\LCM = \max \colon \{1,2\}^\omega \to \{1,2\}$  is not continuous.
	This shows that the natural definition of 
	product does not work.
	
	We now prove that no choice works.
	By way of contradiction, let us suppose that
  one such countable power 
$(p_i \colon (P,v) \to (\{a,b\},u))_{i \in \omega}$ exists.

\smallskip
\noindent
{\it Claim:}  Up to a bijection,  $P$ is the cartesian product $\{a,b\}^\omega$ and each $p_i$ is the projection on the $i$-th coordinate.
	
	\medskip
We define
	\begin{align*}
		h \colon P & \longrightarrow \{a,b\}^\omega\\
		x & \longmapsto (p_i(x))_{i \in \omega},
	\end{align*}
with the intent of showing that  this function is a bijection.
To  prove that $h$ is injective let $x,y \in P$ be such that $h(x) = h(y)$.
	This means that, for every $i \in \omega$, $p_i(x) = p_i(y)$.
	Let $\{*\}$ be a singleton space, equipped 
	with the multiplicity $\{*\} \to \Z_{>0}$ that maps $*$ to $\LCM(v(x),v(y))$.
	The functions 
	\begin{align*}
		f \colon \{*\} & \longrightarrow P & g \colon \{*\} & \longmapsto P\\
		* & \longmapsto x & * & \longmapsto y
	\end{align*}
	are $\Bms$-morphisms, and for all $i \in \omega$
	we have $p_i(f(*)) = p_i(x) = p_i(y) = p_i(g(*))$, i.e., 
	$p_i \circ f = p_i \circ g$.
	From the uniqueness in the universal property of $P$
	it follows that $f = g$, whence $x = y$.
	This proves that $h$ is injective.
	
	To prove that $h$ is surjective
	let  $(x_i)_{i \in \omega}$ be an element of $\{a,b\}^\omega$.
Again let
 $\{*\}$ be a singleton space, equipped with the multiplicity $\{*\} \to \Z_{>0}$ that maps $*$ to the least common multiple of $\{u(x_i) \mid i \in \omega\}$.
As a consequence, for every $i \in \omega$, the function 
$t_i \colon \{*\} \to \{a,b\}$ mapping $*$ to $x_i$ is a morphism.
The universal property of $P$  yields a unique morphism $s \colon \{*\} \to P$ 
such that for every $i \in \omega$\,\, $p_i(s(*)) = x_i$.
	Then $h(s(*)) = (x_i)_{i \in \omega}$ and  $h$ is surjective.

We have just proved 
 that $(p_i \colon P \to \{a,b\})_{i \in \omega}$ is (up to a bijection
	 that  will  henceforth  be ignored) the cartesian product with its canonical projections. Our claim is thus settled.

\smallskip	 
	 
	Since every $p_i$ is continuous, the topology on $P$ is finer than the product topology.
	Since any two comparable compact Hausdorff topologies on the same set are equal, the topology on $P$ is the product topology.
	Since every $p_i$ decreases multiplicities, for all $(x_i)_{i \in \omega} \in P$
	 the number $v((x_i)_{i \in \omega})$ is a multiple of $\LCM\{u(x_i) \mid i \in \omega\}$.
	However, the proof that $h$ is surjective shows that $v((x_i)_{i \in \omega})$ divides $\LCM\{u(x_i) \mid i \in \omega\}$.  As a consequence,  $v((x_i)_{i \in \omega}) = \LCM \{u(x_i) \mid i \in \omega\}$.
	
	To reach a contradiction, we observe that $v$ is not continuous: 
	Indeed, the preimage under $v$ of $\{1\}$ is the singleton $\{(1)_{i \in \omega}\}$, which is not an open subset of $P$.
\end{proof}

\begin{corollary}
\label{c:not-cocomplete}
	The categories $\uSlg$ and $\SMV$ are not cocomplete. For example, they lack the countable copower of $(\Z \times \Z, (1,2))$ and of $\{0,1\} \times \{0, \frac{1}{2}, 1\}$, respectively. \qed
\end{corollary}

\begin{remark}
We have seen that the category $\Bms$ is not complete. 
	However, $\Bms$ can be embedded,  as a full subcategory,
	into other complete categories already considered in the literature (\cite{cigdubmun,cigmar,abbmarspa}).
	To tweak $\Bms$ in order to get a complete category in the style of these papers, one may proceed as follows:
	\begin{itemize}
		\item 
Embed the lattice
(with respect to the divisibility order)
 with bottom $\Z_{>0}$  into a complete lattice. For instance,  
 the complete lattice of supernatural numbers as defined in \cite{cigdubmun}, 
 the complete lattice of additive subgroups of $\mathbb{R}$ containing $1$
 following  \cite{cigmar},
 or the complete lattice of natural numbers  with respect to the divisibility order 
 following  \cite{abbmarspa}.

\smallskip		
		\item 
Equip   this complete lattice with a topology   $\mathcal T$   coarser than the discrete one, that makes the join of any arity a continuous function.
(Thus, for example,  in \cite{abbmarspa},  $\mathcal T$ is the topology on $\Z_{\geq 0}$ whose
closed sets are the finite unions of principal downsets. In \cite{cigdubmun} and \cite{cigmar}, $\mathcal{T}$ is the Scott topology.)
 Consider continuity with respect to   $\mathcal T$.
	\end{itemize}
	
\end{remark}

\medskip
\noindent
As a final remark in this section,
 while finite limits in $\Bms$ are computed 
 as in $\Set$, this is not the case for all existing limits:

\begin{theorem}
\label{theore:forgetful}
	The forgetful functor from $\Bms$ to $\Set$ does not preserve limits.
\end{theorem}

\begin{proof}
	The product in $\Bms$ of countably many copies of a singleton, each copy with a different multiplicity, is the empty boolean multispace.
\end{proof}

\medskip

\section{Colimits in \texorpdfstring{$\Bms$}{Bms}, and limits in \texorpdfstring{$\uSlg$}{uSlg} and  \texorpdfstring{$\SMV$}{SMV}}

In this section we show that the category $\Bms$ of boolean multispaces has finite coproducts but is not finitely cocomplete, as it lacks some pushouts, as well as some coequalizers. 
Therefore, dually, we get that $\uSlg$
and  $\SMV$ have finite products but are not finitely complete, as they lack some pullbacks, as well as some equalizers.

\begin{proposition}[Finite coproducts in $\Bms$]
 \label{proposition:finite-coproducts} \hfill 
	\begin{enumerate}[label = (\roman*)]
	
		\item
		The  coproduct 
		$(X,u) \amalg (X',u')$ in the category $\Bms$ of boolean multispaces
		is given by the disjoint union
		$( X\sqcup X', u\sqcup u')$,  together with the
		natural  multiplicity-preserving 
		injective morphisms of $(X,u)$ and $(X',u')$ into
		$( X\sqcup X', u\sqcup u')$.

 \smallskip
		\item
		The category of boolean multispaces 
		has a (unique up to isomorphism)
		initial object $(X,u)$, namely the empty space equipped 
		with the unique map from $\emptyset$ to $\mathbb{Z}_{>0}$.
	
	\end{enumerate}

\smallskip
	Therefore, $\Bms$ has finite coproducts.
\end{proposition}

\begin{proof}
	(i)
	We first show that $( X\sqcup X', u\sqcup u')$ has the required universal property.
	Consider a boolean multispace $(Y,v)$ and $\Bms$-morphisms $f \colon (X, u) \to (Y,v)$ and $f' \colon (X', u') \to (Y, v)$.
	The continuous map $f \sqcup f' \colon X \sqcup X' \to Y$
	 whose restriction on $X$ is $f$ and whose restriction on $X'$ is $f'$ is the unique map making the following diagram commute:
	\[
	\begin{tikzcd}
		X \arrow[swap]{rd}{f} \arrow[hook]{r}{}& X \sqcup X' \arrow[dashed, "f \sqcup f'" description]{d}{} & X' \arrow{ld}{f'} \arrow[swap,hook']{l}{}\\
		& Y
	\end{tikzcd}
	\]
	Moreover, this map decreases multiplicities,
	 because so do $f$ and $f'$ and   the multiplicity on $X\sqcup X'$ is the disjoint union of the multiplicities on $X$ and $X'$.
	
	(ii)
	This is immediate.
\end{proof}

From \cref{proposition:finite-coproducts}  it follows that the dual categories $\uSlg$ of unital Specker $\ell$-groups and $\SMV$ of Specker MV-algebras have finite products.

We next
 give an explicit description of such finite products, 
 showing that they are computed as in $\Set$.

\begin{proposition}[Finite products in $\uSlg$ and $\SMV$]
	The categories $\uSlg$ of unital Specker $\ell$-groups and $\SMV$ of Specker MV-algebras have finite products.
	They are computed as the usual products of algebras (i.e.,
	they are  computed 
	as in the categories of unital $\ell$-groups and of MV-algebras):
	\begin{enumerate}[label = (\roman*)]
		\item
	The product of two unital Specker $\ell$-groups $(G,u)$ and $(G',u')$ is the cartesian product $(G \times G', (u,u'))$, and similarly for Specker MV-algebras.

\smallskip		
		\item The terminal object in $\uSlg$ is the trivial algebra, and similarly for $\SMV$.
	\end{enumerate}
\end{proposition} 
\begin{proof}
	(i)
	Given two boolean multispaces $(X,u)$ and $(X', u')$, the categorical product of $\SS(X,u)$ and $\SS(X',u')$ is isomorphic to $\SS((X,u) \amalg (X',u'))$, i.e., by \cref{proposition:finite-coproducts}(i), $\SS((X \sqcup X', u \sqcup u'))$, which is the cartesian product of $\SS(X,u)$ and $\SS(X',u')$.
	
	\smallskip
	(ii)
	It is easily seen that the trivial algebra is a unital Specker $\ell$-group.
	For an alternative proof, we
	can note that 
	the terminal object in $\uSlg$ is obtained by applying $\SS$
	 to the initial object of $\Bms$, which, by
	  \cref{proposition:finite-coproducts}, is the empty boolean multispace.
	It follows that the terminal object in $\uSlg$ is the trivial algebra.
\end{proof}

\begin{theorem}
	\label{t:no-pushouts}
	$\Bms$ lacks some pushouts, as well as some coequalizers.
\end{theorem}

\begin{proof}
	Let $(\{*\},2)$ be a singleton space with multiplicity of $*$ equal to $2$.
	Let $(\alpha \Z_{\geq 0},2)$ be the one-point compactification $\alpha \Z_{\geq 0} = \Z_{\geq 0} \cup \{\infty\}$ of the set 
	$\Z_{\geq 0}$, equipped with the multiplicity constantly equal to $2$.
	Let $(\{*\},1)$ be a singleton space with multiplicity of $*$ equal to $1$.
	
	\smallskip	
	We will show that there is no pushout for the following diagram:
	\[
	\begin{tikzcd}
		(\{*\},2) \arrow{r}{\mathrm{id}} \arrow[swap]{d}{* \mapsto \infty} & (\{*\},1)\\
		(\alpha \Z_{\geq 0},2),
	\end{tikzcd}
	\]
	where the function $(* \mapsto \infty) \colon \{*\} \to \alpha \Z_{\geq 0}$ maps $*$ to the accumulation point $\infty$ of $\alpha\Z_{\geq 0}$.
	
	For a proof we will be guided by the  idea  that the natural candidate
	for a pushout would be $(\alpha\Z_{\geq 0}, v)$ where $v(\infty) = 1$ and for every isolated point $x$ we have $v(x) = 2$.  However, it turns out  $v$ is not continuous. 
	Let us transform this motivating idea into a proof.   
	
	By way of contradiction, suppose there is a pushout
	\[
	\begin{tikzcd}
		(\{*\},2) \arrow{r}{\mathrm{id}} \arrow[swap]{d}{* \mapsto \infty} & (\{*\},1) \arrow[dashed]{d}{h}\\
		(\alpha \Z_{\geq 0},2) \arrow[dashed,swap]{r}{g}& (P,v).\arrow[ul, phantom, "\ulcorner", very near start]
	\end{tikzcd}
	\]
	\medskip
	\noindent
	{\it Claim:}  $g$ is a bijection.
	
	We first prove that $g$ is injective.
	Let $(\alpha\Z_{\geq 0}, 1)$ be the boolean multispace whose underlying space is the one-point compactification of $\Z_{\geq 0}$ and whose multiplicity is constantly equal to $1$.
	The identity function $(\alpha\Z_{\geq 0}, 2) \to (\alpha\Z_{\geq 0}, 1)$ is a morphism, 
	and so is  the function $(* \mapsto \infty) \colon (\{*\},1) \to (\alpha\Z_{\geq 0}, 1)$ that maps $*$ to the accumulation point.
	As a consequence, in the following diagram
	the outer square commutes, 
	whence  there is a unique morphism $r \colon (P,v) \to (\alpha\Z_{\geq 0}, 1)$ making the 
	following diagram commute:

	\[
	\begin{tikzcd}
		(\{*\},2) \arrow{r}{\mathrm{id}} \arrow[swap]{d}{* \mapsto \infty} & (\{*\},1) \arrow{d}{h} \arrow[bend left = 1.5 em]{rdd}{* \mapsto \infty}\\
		(\alpha \Z_{\geq 0},2) \arrow[swap]{r}{g} \arrow[swap, bend right = 1em]{rrd}{\mathrm{id}}& (P,v)\arrow[ul, phantom, "\ulcorner", very near start] \arrow["{r}" description, dashed]{rd}\\
		&&(\alpha \Z_{\geq 0},1)
	\end{tikzcd}
	\]
	Since the identity $(\alpha\Z_{\geq 0}, 2) \to (\alpha\Z_{\geq 0}, 1)$ is injective, its factor $g$ is injective.
	
	\smallskip
	We next prove that $g$ is surjective.
	The image $\mathrm{im}(g)$ of $g$ with the restriction $v_\mid$ of $v$ is a boolean multispace.
	The universal property of $P$
	yields  a unique morphism $d \colon (P,v) \to (\mathrm{im}(g),v_\mid)$ making the following diagram commute:
	
	\bigskip
	\[
	\begin{tikzcd}
		(\{*\},2) \arrow{r}{\mathrm{id}} \arrow[swap]{d}{* \mapsto \infty} & (\{*\},1) \arrow{d}{h} \arrow[bend left = 1.5 em]{rdd}{h_\mid}\\
		(\alpha \Z_{\geq 0},2) \arrow[swap]{r}{g} \arrow[swap, bend right = 1em]{rrd}{g_\mid}& (P,v)\arrow[ul, phantom, "\ulcorner", very near start] \arrow["{d}" description, dashed]{rd}\\
		&&(\mathrm{im}(g),v_\mid)
	\end{tikzcd}
	\]

	\noindent
	\medskip
	With $\iota \colon \mathrm{im}(g) \to P$ 
	shorthand for the inclusion map, also the following diagram  commutes:
	\bigskip	
 
	\[
	\begin{tikzcd}
		(\{*\},2) \arrow{r}{\mathrm{id}} \arrow[swap]{d}{* \mapsto \infty} & (\{*\},1) \arrow{d}{h} \arrow[bend left = 1.5 em]{rdd}{h_\mid} \arrow[bend left = 3 em]{dddrr}{h}\\
		(\alpha \Z_{\geq 0},2) \arrow[swap]{r}{g} \arrow[swap, bend right = 1em]{rrd}{g_\mid} \arrow[bend right = 3 em, swap]{rrrdd}{g}& (P,v)\arrow[ul, phantom, "\ulcorner", very near start] \arrow["{d}" description]{rd}\\
		&&(\mathrm{im}(g),v_\mid)\arrow["{\iota}" description, hook]{rd}\\
		&&&(P,v)
	\end{tikzcd}
	\]

	\smallskip
	\noindent
	By the uniqueness in the universal property of the pushout, the composite 
	$\iota \circ d$
	 is the identity on $P$, and hence the inclusion $\iota$ is surjective.
	Therefore, $\mathrm{im}(g) = P$, whence $g$ is surjective.
	Our claim is thus settled.

	\smallskip
	It follows that $g$ is a bijective continuous function between boolean spaces, and hence it is a homeomorphism.
	Therefore, without loss of generality, we may assume the following:

	---$P = \alpha \Z_{\geq 0}$, 
	
	---$g$ is the identity function (possibly not an identity morphism, as the multiplicity $v$ of $P$ may differ from the one of $(\alpha\Z_{\geq 0}, 2)$), 
	and 
	
	---$h$ is
	  the function that maps $*$ to the accumulation point $\infty$ of $\alpha\Z_{\geq 0}$.
	
	\smallskip
	\noindent
	We then have the diagram
	\[
		\begin{tikzcd}
			(\{*\},2) \arrow{r}{\mathrm{id}} \arrow[swap]{d}{* \mapsto \infty} & (\{*\},1) \arrow{d}{* \mapsto \infty}\\
			(\alpha \Z_{\geq 0},2) \arrow[swap]{r}{\mathrm{id}} & (\alpha\Z_{\geq 0},v)\arrow[ul, phantom, "\ulcorner", very near start]
		\end{tikzcd}
	\]
	
	\noindent
	Since $(*\mapsto \infty) \colon (\{*\}, 1) \to (\alpha\Z_{\geq 0}, v)$ decreases multiplicities, $v(\infty) = 1$.
	
	For every $n \in \Z_{\geq 0}$, let $v_n \colon \alpha \Z_{\geq 0} \to \Z_{>0}$ be the function constantly equal to $1$ except on $n$, where it takes the value  $2$.
	The universal property of the pushout yields   a unique morphism $c_n \colon (\alpha\Z_{\geq 0}, v) \to (\alpha\Z_{\geq 0}, v_n)$ making the following diagram commute, and showing that 
	$c_n$ necessarily coincides with the identity function:

	\medskip
	\[
		\begin{tikzcd}
			(\{*\},2) \arrow{r}{\mathrm{id}} \arrow[swap]{d}{* \mapsto \infty} & (\{*\},1) \arrow{d}{* \mapsto \infty} \arrow[bend left = 1.5 em]{rdd}{* \mapsto \infty}\\
			(\alpha \Z_{\geq 0},2) \arrow[swap]{r}{\mathrm{id}} \arrow[swap, bend right = 1em]{rrd}{\mathrm{id}}& (\alpha\Z_{\geq 0},v)\arrow[ul, phantom, "\ulcorner", very near start] \arrow["{c_n}" description, dashed]{rd}\\
			&&(\alpha \Z_{\geq 0},v_n)
		\end{tikzcd}
	\]
	
	\medskip
	\noindent
	For every $n \in \Z_{\geq 0}$, $v(n)$ is a multiple of $2$ since $c_n$ decreases multiplicities, and is a divisor of $2$ since $\mathrm{id} \colon (\alpha\Z_{\geq 0}, 2) \to (\alpha\Z_{\geq 0}, v)$ decreases multiplicities.
	Therefore, $v(n) = 2$.
	
	But then $v$ is not continuous---a contradiction.
	
	\medskip
	\noindent
	This proves that $\Bms$ lacks some pushouts.
	
	\medskip
	Since $\Bms$ has finite coproducts and is not finitely cocomplete, it lacks some coequalizers.
	The  example above shows  that it lacks the coequalizer of the diagram $(\{*\}, 2) \rightrightarrows (\alpha \Z_{\geq 0}, 2) \sqcup (\{*\}, 1)$, where one function maps $*$ to the accumulation point of $\alpha \Z_{\geq 0}$, and the other 
	function maps it  to the unique element $*$ of $\{*\}$.
\end{proof}

\begin{corollary}
	The categories $\uSlg$ and $\SMV$ lack some pullbacks, 
	as well as some equalizers. \hfill{$\Box$}
\end{corollary}

\begin{remark} \label{r:no-pullbacks}
	Dualizing the proof of \cref{t:no-pushouts}, we get that $\uSlg$
	 lacks the pullback of the following diagram:
	\[
	\begin{tikzcd}[column sep = 7em]
		& (\Z,1) \arrow{d}{\text{multiplication by }2}\\
		(C_{\alpha\Z_{\geq 0}},2) \arrow[swap]{r}{\text{evaluation at }\infty}& (\Z, 2)
	\end{tikzcd}
	\]
	Intuitively, the obvious candidate for the pullback (obtained by computing it as in $\Set$) is the set of continuous functions from $\alpha\Z_{\geq 0}$ to $\Z$ whose evaluation at the accumulation point $\infty$ is in $2\Z$.
	However, this unital $\ell$-group is not Specker, 
	because  the only singular elements are the characteristic functions of clopen subsets of $\alpha\Z_{\geq 0}$ not containing $\infty$, and the constant function $2$ is not in the subgroup generated by these.
	
	Expressing the same example in terms of equalizers, 
	we see that
	$\uSlg$ lacks the equalizer of the diagram $(C_{\alpha\Z_{\geq 0}},2) \times (\Z,1) \rightrightarrows (\Z,2)$, where the two maps are
	\begin{align*}
		(C_{\alpha\Z_{\geq 0}},2) \times (\Z,1) & \longrightarrow (\Z,2) & (C_{\alpha\Z_{\geq 0}},2) \times (\Z,1) & \longrightarrow (\Z,2)\\
		(f,a) & \longmapsto f(\infty) & (f,a) & \longmapsto 2a.
	\end{align*}
\end{remark}

\medskip 

\section{\texorpdfstring{Closure properties of $\uSlg$ inside $\ulg$}{Closure properties of uSlg inside ulg}}
\label{section:otto}

We let $ \lg$ and 
$ \ulg$ 
denote the categories respectively given by 
the equational class of $\ell$-groups with their
$\ell$-homomorphisms, and  the category  of unital $\ell$-groups with their
unital $\ell$-homomorphisms.
In this section we are concerned
with the mutual relationships between these categories
and the category  $ \uSlg$.

\smallskip
While being categorically equivalent to the equational class
of MV-algebras,  unital $\ell$-groups are not definable by
equations---in fact, a trivial compactness argument shows that
 the archimedean property of the
strong unit is not even definable in  first-order logic.
The same undefinability property also holds for 
unital Specker $\ell$-groups and 
Specker MV-algebras. 
Since the equational class of  
MV-algebras has  (finite and infinite)   products,
then so does its equivalent category $ \ulg$  of
 unital $\ell$-groups.  
 Naturally,  the  definition of ``$ \ulg$-product''
 does not coincide with  the definition of a cartesian product in the equational class
 $ \lg$   of  $\ell$-groups,
 because  cartesian products of 
 unital $\ell$-groups, qua
 $\ell$-groups, need no longer be unital $\ell$-groups.
The construction of $ \ulg$-products 
is detailed in the proof of the following result:

\begin{proposition}
	\label{proposition:infinite-product-in-UellG}
	The category
	$ \ulg$ of {\it unital $\ell$-groups}
	has infinite products. 
\end{proposition}
 
 \begin{proof} 
	Let $(G_i,u_i)_{i\in I}$ be a family of
	unital $\ell$-groups.  Let 
	$\prod_i G_i$  be the  (cartesian)  product
	of  $\ell$-groups
	in the equational class of  $\ell$-groups.
	Let us restrict  $\prod_i G_i$ to the set $G$ of those elements of 
	$\prod_i G_i$
	whose absolute value is bounded by 
	some positive integer multiple of the element $u=(u_i)_{i\in I}$. 
	Then the element $u$  becomes a unit of the 
	restricted $\ell$-subgroup  $G$.
	Since
	$\Gamma(G,u)\cong
	\prod_i(\Gamma(G_i,u_i)),$
	the unital  $\ell$-group $G$ equipped with the (distinguished strong order) unit $u$ is the categorical product of the family $(G_i,u_i)_{i}$ in the category of unital $\ell$-groups.
\end{proof}

\begin{proposition}
	\label{proposition:subalgebra}
	There is a unital Specker $\ell$-group
	$(G,u)$ and  a unital $\ell$-subgroup $(H,u)$ of $(G,u)$
	that is not  a unital Specker $\ell$-group.
\end{proposition}

\begin{proof}
	Let $(G,u)$ be the $\ell$-group of
	all continuous $\Z$-valued functions on 
	the boolean space 
	$\alpha \Z_{\geq 0}=$
	 the
	one-point compactification  $\{\infty\}\cup \Z_{\geq 0}$ of
	$\Z_{\geq 0} =\{0,1,2,\dots\}$,
	with its distinguished unit $u$
	given by the function constantly equal to  2  over $\alpha\Z_{\geq 0}$.
	Thus  the singular elements are precisely the characteristic functions 
	of the clopen subsets of $\alpha\Z_{\geq 0}$.  See Theorem \ref{theorem:accozzaglia}(iv).
	The constant 1 over $\alpha\Z_{\geq 0}$ is the greatest singular element of $G$.
	As a Specker $\ell$-group, 
	$G$ is generated  (both as a group and as an $\ell$-group)
	by the set  $S_G$ of its singular elements.
	
	Let $(H,u)\subseteq (G,u)$ be the unital $\ell$-group
	generated (as an $\ell$-group) 
	by  the unit $u$ along with the set $S'_G\subseteq S_G$ of
	singular elements of $G$ except those
	having constant value 1 over 
	an open neighborhood of $\infty$.
	In other words,  $H = \{g \in G \mid g(\infty)\in 2\mathbb{Z}\}$.
	(Note that this is the unital $\ell$-group mentioned in \cref{r:no-pullbacks}.)
	The singular elements of $H$ are precisely the characteristic functions of clopen sets not containing $\infty$.
	Then the constant function $2$ is not in the subgroup generated by the singular elements, 
	whence $H$ is not a Specker $\ell$-group.
\end{proof}
 
 A moment's reflection shows that unital $\ell$-groups 
inherit from MV-algebras, via the equivalence 
$\Xi$ of \cite[\S 7]{cigdotmun},  infinite coproducts  (also known as free products).
By contrast we have: 

\begin{proposition}
	\label{proposition:free-three}
	 The infinite coproduct 
	 $\amalg_{i\in I}(S_i,u_i)$ of 
	 unital Specker $\ell$-groups $(S_i,u_i)$ 
	 computed  within the category  $ \ulg$  of 
	 unital $\ell$-groups  need not
	 be a unital Specker $\ell$-group---even assuming   
	that all maps $u_i^\natural$ defined
	in 	\ref{n:natural} have 
	values bounded by a fixed $n$.
\end{proposition}
\begin{proof}
	By way of contradiction let us assume that the infinite coproduct 
	$\amalg_{i\in I}(S_i,u_i)$ of 
	unital Specker $\ell$-groups $(S_i,u_i)$, 
	computed  within the category $\ulg$ of 
	unital $\ell$-groups, is a unital Specker $\ell$-group.
	Then $\amalg_{i\in I}(S_i,u_i)$ 
	is a coproduct in the category $\uSlg$ of unital Specker $\ell$-groups.
	However, by \cref{c:not-cocomplete}, $\uSlg$ lacks the countable copower of $(\Z \times \Z, (1,2))$.
\end{proof}

\medskip
 
For direct limits of algebras we refer to \cite[\S 21]{gra}.

\begin{proposition}
	\label{proposition:no-limits}
	The direct limit (with $\to$-arrows, also known as 
	the  ``directed co\-limit'', \cite[Examples 11.28(4)]{ada})
	of a direct system of 
	Specker MV-algebras  $A_j$,  as given  in the category of
	MV-algebras and their homomorphisms, {\em need not be
	a Specker MV-algebra}---even
	assuming the  boundedness condition for
	the  homogeneity degree of all 
	homogeneous components of all $A_j$.
\end{proposition}

\begin{proof}
	If the directed colimit of  a direct system of 
	Specker MV-algebras $A_j$,  as given  in the category of
	MV-algebras and their homomorphisms,  were a Specker MV-algebra, then it would be a directed colimit in the category of Specker MV-algebras.
	
	Recall that the category of Specker MV-algebras is finitely cocomplete (\cref{corollary:finitely-cocomplete}), but that $\{0,1\} \times \{0, \frac{1}{2}, 1\}$  lacks a countable copower.
	Let $\mathsf{I}$ be the category of finite subsets of $\Z_{\geq 0}$ and inclusions between them.
	Let $F \colon \mathsf{I} \to \SMV$ be the functor that, on objects, maps a finite subset $S$ to the $S$-fold copower of $\{0,1\} \times \{0, \frac{1}{2}, 1\}$ and, on morphisms,  maps an inclusion $S \subseteq T$ to the canonical morphism of the $S$-fold copower of $\{0,1\} \times \{0, \frac{1}{2}, 1\}$ into the $T$-fold copower of $\{0,1\} \times \{0, \frac{1}{2}, 1\}$.
	It follows that this directed diagram has no colimit.
	For otherwise, by the dual of the exercise
	``Products as Projective Limits of Finite Products'' \cite[Exercise~11B]{ada}, it would give a countable copower of $\{0,1\} \times \{0, \frac{1}{2}, 1\}$, a contradiction.
\end{proof}

\section{Related work}
\label{section:related-work}

Multisets and their topological variants
(msets, bags, heaps, bunches, $M$-to\-po\-lo\-gic\-al spaces,
weighted sets, firesets, etc.)
have
an extensive literature. See, e.g.,  \cite{bli},
\cite{gir}, and references therein.

For an earlier paper on the relationships between  multisets  and 
MV-algebras see 
\cite{cigdubmun}, which contains a duality for locally finite MV-algebras, of which Specker MV-algebras are special cases.
For a more general result see
\cite{cigmar}, which constructs a duality for ``locally weakly finite'' 
MV-algebras. The latter class constitutes
a generalization of locally finite MV-algebras that includes, 
among others, the standard MV-algebra $[0,1]$.

For recent work  
see  \cite{abbmarspa}, where a duality is constructed for the
class of ``metrically complete unital $\ell$-groups''.
Every Specker MV-algebra is locally finite, and  
every unital Specker $\ell$-group is metrically complete.
Dually, every boolean multispace
$(X, u \colon X \to \Z_{>0})  \,\,\,$ 
is  a multiset in the sense of \cite{cigdubmun} (up to identifying $\Z_{>0}$ with a subset of the set of supernatural numbers),   a ``normal a-space'' in the sense of \cite{abbmarspa},
and  a real-valued multiset $(X, u' \colon X \to \mathrm{Sub}([0,1]))$
in the sense of \cite{cigmar} (up to identifying $n \in \Z_{>0}$ 
with $\{0, \frac{1}{n}, \dots, \frac{n-1}{n}, 1\}$).


\end{document}